\documentclass{amsart}

\usepackage[T1]{fontenc}
\usepackage[latin10]{inputenc}
\usepackage{amssymb, amsmath, amsthm}
\usepackage{hyperref, amsrefs}
\usepackage{enumerate}
\usepackage{verbatim}

\def\dd{\mathcal D}
\def\sd{\sum_{I \in \dd}}
\def\fd{\mathbf f}
\def\Fd{\mathbf F}
\def\gd{\mathbf g}
\def\Gd{\mathbf G}
\def\fb{\mathcal B}
\def\ud{\mathbf U}
\def\vd{\mathbf V}
\def\dw{\mathbf W}
\def\md{\mathbf M}

\def\ltr{L^2(\mathbb{R})}
\def\ltrp{L^2(\mathbb{R}^p)}
\def\ltw{L^2(W)}
\def\ltwi{L^2(W^{-1})}

\newtheorem{theorem}{Theorem}[section]

\newtheorem{lemma}[theorem]{Lemma}

\newtheorem{remark}[theorem]{Remark}

\title{Bounds for Calder\'{o}n-Zygmund operators with matrix \(A_2\) weights}

\author{Sandra Pott}
\address{Centre for Mathematical Sciences, University of Lund, P.O. Box 118, SE-221 00 Lund, Sweden}
\email{sandra@maths.lth.se}
\author{Andrei Stoica}
\address{Centre for Mathematical Sciences, University of Lund, P.O. Box 118, SE-221 00 Lund, Sweden}
\email{andrei@maths.lth.se}


\begin{document}

\maketitle

\begin{abstract}
It is well-known that dyadic martingale transforms are a good model for Calder\'{o}n-Zygmund singular integral operators. In this paper we extend some results on weighted norm inequalities to vector-valued functions. We prove that, if \(W\) is an \(A_2\) matrix weight, then the weighted \(L^2\)-norm of a Calder\'{o}n-Zygmund operator with cancellation has the same dependence on the \(A_2\) characteristic of \(W\) as the weighted \(L^2\)-norm of an appropriate matrix martingale transform. Thus the question of the dependence of the norm of 
matrix-weighted Calder{\'o}n-Zygmund operators on the \(A_2\) characteristic of the weight is reduced to the case of dyadic martingales and paraproducts. We also show a slightly different proof for the special case of Calder\'{o}n-Zygmund operators with even kernel, where only scalar martingale transforms are required. We conclude the paper by proving a version of the matrix-weighted Carleson Embedding Theorem.

Our method uses a Bellman function technique introduced by S. Treil to obtain the right estimates for the norm of dyadic Haar shift operators. We then apply the representation theorem of T. Hyt\"{o}nen to extend the result to general Calder\'{o}n-Zygmund operators.

\end{abstract}

\section{Introduction}

In the 1970's, R.A. Hunt, B. Muckenhoupt and R.L. Wheeden \cite{HuMuWh73} and R.R. Coifman and C. Fefferman \cite{CoFe74} showed that a Calder\'{o}n-Zygmund singular integral operator is bounded on the weighted space \(L^p(w)\) if and only if the scalar weight \(w\) belongs to the so-called \(A_p\) class. For the last two decades, an important open problem in Harmonic Analysis was to characterize the dependence of the operator norm on the \(A_p\) characteristic, \([w]_{A_p}\), of the weight. For \(p=2\) this dependence was conjectured to be linear in \([w]_{A_2}\); the problem has become known as the \(A_2\) conjecture. The first step was taken by J. Wittwer \cite{Wi00}, who proved the \(A_2\) conjecture for dyadic martingale transforms. Using Bellman function techniques, S. Petermichl and A. Volberg \cite{PeVo02} showed the conjecture for the Beurling-Ahlfors transform. It took a few more years until the \(A_2\) conjecture was proved for the Hilbert transform by S. Petermichl (see \cite{Pe07}). The conjecture was finally settled for general Calder\'{o}n-Zygmund operators in 2010 by T. Hyt\"{o}nen \cite{Hy12a}. The main ingredient in his proof is the pointwise representation of a general Calder\'{o}n-Zygmund operator as a weighted average over an infinite number of randomized dyadic systems of some simpler operators (called dyadic Haar shifts) in such a way that the estimates for the dyadic Haar shifts depend polynomially on the complexity. 

A natural problem is to try to extend these results to vector-valued functions. S. Treil and A. Volberg introduced the correct definition of a matrix \(A_p\) weight (see \cite{TrVo97}). M. Goldberg \cite{Go03}, F. Nazarov and S. Treil \cite{hunt} and A. Volberg \cite{Vo97} showed that certain Calder\'{o}n-Zygmund operators are bounded on \(L^p(W)\) when \(1<p<\infty\) if \(W\) is a matrix \(A_p\) weight. However, the sharp dependence of the norm of a Calder\'{o}n-Zygmund operator on the \(A_2\) characteristic of \(W\) is unknown even for the martingale transform. In a recent paper, K. Bickel, S. Petermichl and B. Wick \cite{BiPeWi14} modified a scalar argument to obtain that for the Hilbert and martingale transforms this dependence is no worse than \([W]_{A_2}^{3/2}\log{[W]_{A_2}}\). This has very recently been improved to
 \([W]_{A_2}^{3/2}\), or more precisely,  \([W]_{A_2}^{1/2} [W]_{A_\infty}^{1/2} [W^{-1}]_{A_\infty}^{1/2}\), for all  Calder\'{o}n-Zygmund operators  \cite{ntvp}. 
 
 Even more recently, T. Hyt\"onen, S. Petermichl and A. Volberg \cite{hpv} proved
  the sharp linear upper bound 
 \([W]_{A_2} \) for the matrix-weighted square function, which can be understood as an average of the matrix martingale transforms
 we consider. This raises the hope that the expected sharp linear bound for matrix martingale transforms in terms of  $[W]_{A_2}$
 may now come into reach.
 
In this paper we prove that the norms of all Calder\'{o}n-Zygmund singular integrals with cancellation have the same dependence on \([W]_{A_2}\) as the matrix martingale transforms 
(we denote this dependence by \(N([W]_{A_2})\)). The \(A_2\) conjecture for matrix-weighted spaces is thus reduced to the case of dyadic martingale transforms and of the paraproducts. The proof follows S. Treil's approach for the proof of the linear \(A_2\) bound in the scalar case (see \cite{Tr11}). The main challenge here is the adaptation of the Bellman function to the matrix case, where convexity properties are much more difficult than in the scalar setting. Using Hyt\"{o}nen's representation of a Calder\'{o}n-Zygmund operator, it is enough to obtain the right estimate for the dyadic Haar shift operators. Since we want to obtain the same bound in terms of \([W]_{A_2}\) for the norm of dyadic Haar shifts, we have to use the martingale transform only once. We will decompose a dyadic Haar shift of complexity \(k\) into \(k\) ``slices'' that can be seen as martingale transforms. The main idea is to linearize the norm of these slices and then use the Bellman function to estimate each summand. In order to do this, we start with a standard dyadic martingale of points from the domain of the Bellman function, where at each point we have two choices with equal probability. We will then modify the martingale, but preserving the initial point and the endpoints, and probabilities. From the starting point, instead of going to the next level in the standard martingale, we move with probabilities \(1/2\) to two new points that are ``far enough'' from the initial point, but also ``almost averages'' of the endpoints. We can still move from these new points to the endpoints, this time using a modified dyadic martingale, where at each point we have two choices with ``almost equal'' probability. This new martingale is constructed in such a way that the probabilities of moving from the starting point to the endpoints are still equal, as in the case of the standard martingale. Although we have used probabilistic terms, the formal proof involving the Bellman function is elementary.

The paper is organized as follows: in Section 2 we recall the necessary definitions and results that we are using. Then we state our main result (Theorem \ref{mainthm}) and show that it is enough to obtain a corresponding estimate for dyadic Haar shift operators, which is the content of Theorem \ref{mainshift}. In Section 3 we use the boundedness of the martingale transform to relate the norm of a dyadic Haar shift to an expression that will be controlled by the Bellman function. Section 4 contains the definition of the Bellman function associated to our problem and the description of its properties. In Section 5 we formulate and prove the main technical result of the paper, which is inspired by \cite{Tr11}. In Section 6 we show how the main estimate from the previous section is used to conclude the proof of Theorem \ref{mainshift}. In the following section we prove a similar result for Calder\'{o}n-Zygmund singular integrals with even kernel, this time using the same martingale transform as in the scalar case. We finish with a further application of our Bellman function argument, namely a matrix-weighted Carleson Embedding Theorem which holds with constants independent of the dimension and the weight. This is, however, not the simple generalization of the usual weighted Carleson Embedding Theorem in \cite{NaTrVo99}.

\section{Definitions and statement of the main results}

In this section, we recall some well-known notions and results that we are going to use later on.

\subsection{Calder\'{o}n-Zygmund operators}

Let \(\Delta=\{(x,x): x \in \mathbb{R}^p\}\) be the diagonal of \(\mathbb{R}^p \times \mathbb{R}^p\). We say that a function \(K: \mathbb{R}^p \times \mathbb{R}^p \setminus \Delta \to \mathbb{C}\) is a standard Calder\'{o}n-Zygmund kernel if there exists \(\delta>0\) such that
\[|K(x,y)| \leq \frac{C}{|x-y|^p},\]
\[|K(x,y)-K(x,z)|+|K(y,x)-K(z,x)| \leq C_{\delta} \frac{|y-z|^{\delta}}{|x-y|^{p+\delta}},\]
for all \(x,y,z \in \mathbb{R}^p\) with \(|x-y|>2|y-z|\).

An operator \(T\), defined on the class of step functions (which is dense in \(L^2(\mathbb{R}^p)\)), is called a Calder\'{o}n-Zygmund operator on \(\mathbb{R}^p\) associated to \(K\), if it satisfies the kernel representation
\[Tf(x)=\int_{\mathbb{R}^p} K(x,y)f(y)\, \mathrm{d}y, \qquad x \notin {\mathrm{supp}}\,f .\]

\subsection{Matrix \(A_2\) weights}

For \(d \geq 1\), the non-weighted Lebesgue space \(\ltrp\) consists of all measurable functions \(f:\mathbb{R}^p \to \mathbb{C}^d\) such that
\[
\|f\|_{\ltrp} := \Big( \int_{\mathbb{R}^p}\,\|f(t)\|^2_{\mathbb{C}^d}\, \mathrm{d}t \Big)^{1/2} < \infty.
\] 
We will also use the space \(C^1_c(\mathbb{R}^p)\) of compactly supported, continuously differentiable functions \(f:\mathbb{R}^p \to \mathbb{C}^d\).

Let \(\mathcal{M}_d(\mathbb{C})\) be the space of \(d \times d\) complex matrices. A matrix weight on \(\mathbb{R}^p\) is a measurable locally integrable function \(W: \mathbb{R}^p \to \mathcal{M}_d(\mathbb{C}) \) whose values are almost everywhere positive definite. We define \(\ltw\) to be the space of measurable functions \(f:\mathbb{R}^p \to \mathbb{C}^d\) with norm 
\[\|f\|^2_{\ltw} = \int_{\mathbb{R}^p} \|W^{1/2}(t)f(t)\|^2_{\mathbb{C}^d}\, \mathrm{d}t = \int_{\mathbb{R}^p} \langle W(t)f(t), f(t) \rangle \, \mathrm{d}t < \infty.\]
It is well-known that the dual of \(\ltw\) can be identified with  \(\ltwi\), where the duality between these two spaces is given by the unweighted standard inner product.

We say that a matrix weight \(W\) satisfies the matrix \(A_2\) Muckenhoupt condition if 
\begin{equation}\label{muckenhoupt}
[W]_{A_2} := \sup_{Q} \bigg \| \Big(\frac{1}{|Q|} \int_{Q} W(t)\, \mathrm{d}t \Big)^{1/2} \Big(\frac{1}{|Q|} \int_{Q} W^{-1}(t)\, \mathrm{d}t \Big)^{1/2} \bigg \| < \infty,
\end{equation}
where the supremum is taken over all cubes \(Q \subset \mathbb{R}^p\), and \(\|\cdot\|\) denotes the norm of the matrix acting on \(\mathbb{C}^d\). 
The number \([W]_{A_2}\) is called the \(A_2\) characteristic of the weight \(W\). We say that a matrix weight \(W\) satisfies the dyadic matrix Muckenhoupt condition \(A_2^d\) on \(\mathbb{R}^p\) or \(\mathbb{R}\), if \eqref{muckenhoupt} is satisfied, but with the supremum now being taken only over dyadic cubes or intervals, respectively  (see \cite{TrVo97}).

\subsection{Dyadic setting}

Since we will reduce the proof of our main result to the case of functions defined on \(\mathbb{R}\), we will only introduce the required notions in this setting. For the analogous definitions in the case of functions on \(\mathbb{R}^p\), we refer the readers to \cite{Hy11}.

The standard dyadic system in \(\mathbb{R}\) is
\[\dd^0:=\bigcup_{j \in \mathbb{Z}} \dd^0_j, \qquad \dd^0_j:=\{2^{-j}([0,1)+k): k \in \mathbb{Z}\}.\]
Given a binary sequence \(\omega=(\omega_j)_{j \in \mathbb{Z}} \in (\{0,1\})^{\mathbb{Z}}\), a general dyadic system on \(\mathbb{R}\) is defined by
\[\dd^{\omega}:=\bigcup_{j \in \mathbb{Z}} \dd^{\omega}_j, \qquad \dd^{\omega}_j:= \dd^0_j + \sum_{i>j} 2^{-i} \omega_i.\]
When the particular choice of \(\omega\) is not important, we will use the notation \(\dd\) for a generic dyadic system. We equip the set \(\Omega:=(\{0,1\})^{\mathbb{Z}}\) with the canonical product probability measure \(\mathbb{P}_{\Omega}\) which makes the coordinates \(\omega_j\) independent and identically distributed with \(\mathbb{P}_{\Omega}(\omega_j=0)=\mathbb{P}_{\Omega}(\omega_j=1)=1/2\). We denote by \(\mathbb{E}_{\Omega}\) the expectation over the random variables \(\omega_j, j \in \mathbb{Z}\).

For an interval \(I \in \dd\), let \(I^+\) and \(I^-\) be the left and right children of \(I\). The parent of \(I\) will be denoted by \(\tilde{I}\). We will also use the notation \[\dd_n(I):=\{J \in \dd: J \subset I, |J|=2^{-n}|I|\}\]
for the collection of \(n\)-th generation children of \(I\), where \(|J|\) stands for the length of the interval \(J\).

For any interval \(I \in \dd\), there is an associated Haar function defined by
\[h_I=|I|^{-1/2}(\chi_{I^+}-\chi_{I^-}),\]
where \(\chi_I\) is the characteristic function of \(I\).

For an arbitrary dyadic system \(\dd\), the Haar functions form an orthogonal basis of \(\ltr\). Hence any function \(f \in \ltr\) admits the orthogonal expansion
\[f=\sd \langle f,h_I \rangle h_I.\]
We denote the average of a locally integrable function \(f\) on the interval \(I\) by \(\langle f \rangle_I :=|I|^{-1}\int_I f(t)\, \mathrm{d}t\).

Let \(W\) be a matrix weight. For a sequence of \(d \times d\) matrices \(\sigma=\{\sigma_I\}_{I \in \dd}\), we introduce the notation \( \|\sigma\|_{\infty, W} = \sup_{I \in \dd} \big \| \langle W \rangle_I ^{1/2} \sigma_I \langle W \rangle_I ^{-1/2} \big \| \).

For a sequence \(\sigma\) such that \( \|\sigma\|_{\infty, W} < \infty\), we define the martingale transform operator \(T_{\sigma}\) by
\[T_\sigma f = \sd \sigma_I  \langle f,h_I \rangle h_I.\]
If $W$ is a matrix $A_2$ weight, then the condition \( \|\sigma\|_{\infty, W} < \infty\) is equivalent to the boundedness of \(T_{\sigma}\) on \(\ltw\) (see, e.g.
 Theorem 5.2 in \cite{BiPeWi14} for an explicit statement; it is also contained in \cite{TrVo97}). Such martingale transforms are considered a good model for Calder\'{o}n-Zygmund singular integral operators. 

A (cancellative) dyadic Haar shift on \(\mathbb{R}\) of parameters \((m,n)\), with \(m,n \in \mathbb{N}_0\), is an operator of the form
\[S f = \sum_{L \in \dd}
\sum_{\substack{
            I \in \mathcal{D}_m(L) \\
            J \in \mathcal{D}_n(L)}}
c^{L}_{I,J} \langle f, h_I \rangle h_J,\]
where \(\left |c^{L}_{I,J}\right| \leq \frac{\sqrt{|I|} \sqrt{|J|}}{|L|} =2^{-(m+n)/2} \) and \(f\) is any locally integrable function. The number \(k := \max\{m,n\}+1\) is called the complexity of the Haar shift.

For \(0 \leq j \leq k-1\) we introduce the notation \(\mathcal{L} _j :=  \{ I \in \mathcal{D} : |I|=2^{j+kt},  t \in \mathbb{Z} \},\) and define the slice \(S_j\) by \[S_j f = \sum_{L \in \mathcal{L}_j}
\sum_{\substack{
            I \in \mathcal{D}_m(L) \\
            J \in \mathcal{D}_n(L)}}
c^{L}_{I,J} \langle f, h_I \rangle h_J. \]
We can thus decompose \(S\) as \(S=\sum _{j=0} ^{k-1} S_j\). The key point is now that the operators \(S_j\) can be seen as martingale transforms when we are moving \(k\) units of time at once, so it is possible to apply the Bellman function for dyadic martingale transforms.

Following the approach in \cite{Tr11}, one can show that it is enough to consider only dyadic Haar shifts on a dyadic system in \(\mathbb{R}\). The following construction works for general dyadic systems, but for convenience we will assume that we are dealing with the standard one. This reduction is obtained by ``arranging'' the dyadic cubes on the real line. 

More precisely, for a dyadic cube \(Q\) in \(\mathbb{R}^p\), we choose a dyadic interval \(I\) such that \(|I| = |Q|\) (this interval \(I\) will correspond to the cube \(Q\)). We then split \(Q\) into two congruent parallelepipeds by dividing one of its sides into two parts, and then pick a bijection between these two parallelepipeds and the children of \(I\). By dividing a long side, we split each parallelepiped into two congruent ones, and then choose a bijection between the four parallelepipeds and the children of the two intervals from the previous step. After \(p\) divisions we obtain a bijection between the children of \(Q\) and the intervals \(J \in \dd_p(I)\). The intervals \(J \in \dd_n(I), 1 \leq n < p\), correspond to some ``almost children'' \(R\) of \(Q\), where by an ``almost child'' of \(Q\) we mean a parallelepiped with some sides coinciding with the sides of \(Q\), and the other sides being half of the corresponding sides of \(Q\).

This construction can also be done in the opposite direction. If \(\tilde{I}\) is the parent of the interval \(I\), and \(\tilde{Q}\) is the grandparent of \(Q\) of order \(p\), by the above method we obtain a bijection \(\Phi\) between the children and ``almost children'' of \(\tilde{Q}\), and the intervals \(J \in \dd_n(\tilde{I}), 1 \leq n \leq p\), such that \(\Phi(Q) = I\). To make sure that \(\Phi(Q) = I\), at each division we have to assign to the "almost child" containing \(Q\) the dyadic interval of appropriate length that contains \(I\). 

A locally integrable function \(f\) on \(\mathbb{R}^p\) will thus be transferred to a locally integrable function \(g\) on \(\mathbb{R}\) such that \(\langle f \rangle_Q = \langle g \rangle_I\), for all \(Q\) and \(I\) with \(\Phi(Q) = I\).

We now look at the differences that arise when using this reduction. If \(S\) is a dyadic Haar shift (or one of its slices) of complexity \(k\) in \(\mathbb{R}^p\), then its model in \(\mathbb{R}\) will be a Haar shift of complexity \(kp\).

If \(W\) is a matrix \(A_2^d\) weight on \(\mathbb{R}^p\), then the \(A_2^d\) characteristic of the corresponding weight on \(\mathbb{R}\) is \(\sup_{R} \big \| \langle W \rangle_R^{1/2} \langle W^{-1} \rangle_R ^{1/2} \big \|\), where the supremum is taken over all dyadic cubes in \(\mathbb{R}^p\) and all their ``almost children''. If \(R\) is an ``almost child'' of a cube \(Q\), then
\[\int_{R} W(t)\, \mathrm{d}t \leq \int_{Q} W(t)\, \mathrm{d}t, \quad \int_{R} W^{-1}(t)\, \mathrm{d}t \leq \int_{Q} W^{-1}(t)\, \mathrm{d}t,\]
and \(|R| \geq 2^{-p+1} |Q|\). We thus have 
\begin{align*}
\big \| \langle W \rangle_R^{1/2} \langle W^{-1} \rangle_R^{1/2} \big \|^2 & = \big \| \langle W \rangle_R^{1/2} \langle W^{-1} \rangle_R \langle W \rangle_R^{1/2} \big \| \leq \big \| \langle W \rangle_R^{1/2} 2^{p-1} \langle W^{-1} \rangle_Q \langle W \rangle_R^{1/2} \big \| \\
& = 2^{p-1} \big \| \langle W^{-1} \rangle_Q^{1/2} \langle W \rangle_R \langle W^{-1} \rangle_Q^{1/2} \big \| \leq 2^{p-1}  \big \| \langle W^{-1} \rangle_Q^{1/2} 2^{p-1} \langle W \rangle_Q \langle W^{-1} \rangle_Q^{1/2} \big \| \\
& = 2^{2(p-1)} \big \| \langle W \rangle_Q^{1/2} \langle W^{-1} \rangle_Q^{1/2} \big \|^2.
\end{align*}
Thus, after the transfer to the real line, the \(A_2^d\) characteristic \([W]_{A_2^d}\) of the weight increases at most by a factor of \(2^{2(p-1)}\).
 
We are using the following representation of a Calder\'{o}n-Zygmund operator in terms of dyadic Haar shifts.

\begin{theorem} [\protect{Hyt\"{o}nen \cite{Hy11}}] \label{repr}
Let \(T\) be a Calder\'{o}n-Zygmund operator on \(\mathbb{R}^p\) which satisfies the standard kernel estimates, the weak boundedness property \(|\langle T \chi_Q, \chi_Q\rangle | \leq C|Q|\) for all cubes \(Q\), and the vanishing paraproduct conditions  \(T(1)=T^*(1)=0\). Then it has an expansion, say for \(f,g \in C^1_c(\mathbb{R}^p)\),
\[\langle T f,g \rangle_{L^2(\mathbb{R}^p),L^2(\mathbb{R}^p)} = 
C \cdot \mathbb{E}_{\Omega} \sum_{m,n=0}^{\infty} \tau(m,n) \langle S^{mn}_{\omega} f, g \rangle_{L^2(\mathbb{R}^p),L^2(\mathbb{R}^p)},\]
where \(C\) is a constant depending only on the constants in the standard estimates of the kernel \(K\) and the weak boundedness property, \(S^{mn}_{\omega}\) is a dyadic Haar shift in \(\mathbb{R}^p\) of parameters \((m,n)\) on the dyadic system \(\mathcal{D}^{\omega}\), and \(\tau(m,n) \lesssim P(\max\{m,n\}) 2^{-\delta \max\{m,n\}} \), with \(P\) a polynomial.
\end{theorem}

We define the function \(N:[1, \infty) \to [1, \infty)\) by
\[ N(X)=  \sup \|T_{\sigma}\|_{\ltw \to \ltw}, \]
where the supremum is taken over all \( d \times d\) matrix \(A_2^d\) weights \(W\) with \([W]_{A_2^d} \leq X\) and all matrix sequences $\{\sigma\}_{I \in \mathcal{D}}$ with
\(\|\sigma\|_{\infty, W} \leq 1\). It was shown in \cite{BiPeWi14} that
\begin{equation} \label{cubicbound}
N(X) \lesssim (\log X ) X^{3/2}.
\end{equation}

Here is our main result:

\begin{theorem}\label{mainthm}
Let \(W\) be a \(d \times d\) matrix \(A_2\) weight on \(\mathbb{R}^p\). 
Let \(K\) be a standard kernel and \(T\) be a Calder\'{o}n-Zygmund operator on \(\mathbb{R}^p\) associated to \(K\). Suppose that \(T\) satisfies the weak boundedness property \(|\langle T \chi_Q, \chi_Q \rangle | \leq C|Q|\) for all cubes \(Q\), and the vanishing paraproduct conditions \(T(1)=T^*(1)=0\). Then
\[\|T\|_{\ltw \to \ltw} \leq C \cdot p d N(2^{2(p-1)} [W]_{A_2}) \leq C_p \cdot d N([W]_{A_2}),\]
where \(C\) depends only on the constants in the standard estimates and the weak boundedness property, while \(C_p\) depends on \(C\) and \(p\).
\end{theorem}

The second inequality in the theorem is a simple consequence of (\ref{cubicbound}),
we therefore turn to the first inequality.  It is enough to show a corresponding result for Haar shift operators and then use the representation theorem of T. Hyt\"{o}nen. 

Let \(f,g \in C_c^1(\mathbb{R}^p)\)  (if \(W\) is a matrix \(A_2\) weight, this space is dense in both \(\ltw\) and \(\ltwi\)). Since the duality between \(\ltw\) and \(\ltwi\) is the same as the standard duality on \(\ltrp\),
by Theorem \ref{repr} we have the representation
\[\langle Tf, g \rangle_{\ltw,\ltwi} = C \cdot \mathbb{E}_{\Omega} \sum_{m,n=0}^{\infty} \tau(m,n) \langle S^{mn}_{\omega} f, g \rangle_{\ltw,\ltwi}\]
and therefore
\[ \|T\|_{\ltw \to \ltw} \leq C  \sum_{m,n=0}^{\infty} \tau(m,n) \|S^{mn}_{\omega}\|_{\ltw \to \ltw}.\]
We will show the estimate 
\[
    \|S^{mn}\|_{\ltw \to \ltw} \lesssim (\max\{m,n\}+1) p d N(2^{2(p-1)} [W]_{A_2})
    \]
     for all dyadic Haar shifts \(S^{mn}\) on \(\mathbb{R}^p\) with parameters \((m,n)\), which ensures the convergence of the series and completes the proof of Theorem \ref{mainthm}. Using the above transference result, we can restrict ourselves to Haar shifts in \(\mathbb{R}\). This is the content of the following theorem.

\begin{theorem}\label{mainshift}
Let \(S\) be a dyadic Haar shift on \(\mathbb{R}\) of complexity \(k \geq 1\) and \(W\) be a matrix \(A_2^d\) weight. Then
\[\|S\|_{\ltw \to \ltw} \leq c \cdot k d N([W]_{A_2^d}),\]
where c is an absolute, positive constant.
\end{theorem}

\section{Reduction of the proof of Theorem \ref{mainshift} }

Let \(W\) be a \(d \times d\) matrix \(A_2^d\) weight on \(\mathbb{R}\). For each \(I \in \dd\), choose an orthonormal basis of eigenvectors  \(B_I = \{e_I^1, e_I^2, \ldots, e_I^d\}\) of \( \langle W \rangle_I\), and let \(P_I^i, 1 \leq i \leq d\), be the corresponding orthogonal projection onto the span of \(e_I^i\). 

Using the definition of the martingale transform operator \(T_{\sigma}\) and the fact that each \(\langle W\rangle_I\) commutes with the \(P_I^i\)'s, we have for $f \in \ltw$, $g \in \ltwi$,
\begin{align} \label{est:projections}
& \sum_{i=1}^d \sd  \big|   \big \langle P_I^i \langle f,h_I \rangle, P_I^i \langle g,h_I \rangle \big \rangle_{\mathbb{C}^d} \big| \\ \nonumber
& \le d \cdot \sup_{\sigma}
\sd \big \langle \sigma_I \langle f,h_I \rangle, \langle g,h_I \rangle \big \rangle_{\mathbb{C}^d} \\   \nonumber
& = d \cdot \sup_{\sigma}
\langle T_\sigma f,g \rangle_{\ltw,\ltwi} \\\nonumber
&\le d \cdot \sup_{\sigma} \|T_{\sigma}\|_{\ltw \to \ltw}        \|f \|_{\ltw} \|g\|_{\ltwi} \\   \nonumber
&\le d \cdot N([W]_{A_2^d})   \|f \|_{\ltw} \|g\|_{\ltwi}, \nonumber
\end{align}
where the supremum is now taken over all matrix sequences \(\sigma = \{\sigma_I\}_{I \in \dd}\) such that \(\| \sigma\|_{\infty, W} \le 1\). Notice that it would suffice to just take the $\sigma_I$'s which are diagonal in the basis $B_I$.

We can thus rewrite the estimate (\ref{est:projections}) above as
\begin{multline} \label{linearisation}
 \sum_{i=1}^d \sd \left | \big \langle  P_I^i \big( \langle f \rangle _{I^+} - \langle f \rangle_{I^-} \big), P_I^i \big( \langle g \rangle _{I^+} - \langle g\rangle_{I^-} \big) \big \rangle_{\mathbb{C}^d} \right | \cdot |I| \\
 = 4 \sum_{i=1}^d \sd \left | \big \langle  P_I^i \langle f,h_I \rangle,  P_I^i \langle g,h_I \rangle \big \rangle_{\mathbb{C}^d} \right |  
 \leq 4 \, d \cdot N([W]_{A_2^d}) \|f\|_{\ltw} \|g\|_{\ltwi}
\end{multline}
for all \(f \in \ltw\) and \( g \in \ltwi.\)

Since \(S\) is a Haar shift operator of complexity \(k\), it has the form
\[S f = \sum_{L \in \dd}
\sum_{\substack{
            I \in \mathcal{D}_m(L) \\
            J \in \mathcal{D}_n(L)}}
c^{L}_{I,J} \langle f, h_I \rangle h_J,\]
where \(\left |c^{L}_{I,J}\right| \leq \frac{\sqrt{|I|} \sqrt{|J|}}{|L|} =2^{-(m+n)/2} \).

Let \(f \in \ltw,\  g \in \ltwi\) and \(0 \leq j \leq k-1\) be fixed. For the slice \(S_j\), we can write
\begin{align*}
& \left \langle S_j f,g \right \rangle_{\ltw,\ltwi}  = \Bigg \langle \sum_{L \in \mathcal{L}_j}
\sum_{\substack{
            I \in \mathcal{D}_m(L) \\
            J \in \mathcal{D}_n(L)}}
c^{L}_{I,J} \langle f, h_I \rangle h_J, \sum_{I' \in \dd} \langle g, h_{I'} \rangle h_{I'} \Bigg \rangle _{\ltw,\ltwi}\\
& \qquad = \sum_{L \in  \mathcal{L}_j} \sum_{I' \in \dd}
\sum_{\substack{
            I \in \mathcal{D}_m(L) \\
            J \in \mathcal{D}_n(L)}}
c^{L}_{I,J} \big \langle \langle f, h_I \rangle , \langle g, h_{I'} \rangle \big \rangle _{\mathbb{C}^d} \langle h_J,h_{I'} \rangle_{L^2(\mathbb{R}),L^2(\mathbb{R})} \\
& \qquad = \sum_{L \in  \mathcal{L}_j}
\sum_{\substack{
            I \in \mathcal{D}_m(L) \\
            J \in \mathcal{D}_n(L)}}
c^{L}_{I,J} \big \langle \langle f, h_I \rangle , \langle g, h_J \rangle \big \rangle _{\mathbb{C}^d} 
= \sum_{L \in  \mathcal{L}_j} \sum_{i = 1}^d  
\sum_{\substack{
            I \in \mathcal{D}_m(L) \\
            J \in \mathcal{D}_n(L)}}
c^{L}_{I,J} \big \langle  P_L^i \langle f, h_I \rangle , P_L^i \langle g, h_J \rangle \big \rangle _{\mathbb{C}^d} \\
& \qquad = \sum_{L \in  \mathcal{L}_j} \sum_{i = 1}^d  
\sum_{\substack{
            I \in \mathcal{D}_m(L) \\
            J \in \mathcal{D}_n(L)}}
c^{L}_{I,J} \frac{|I|^{1/2}}{2^{k-m}} \frac{|J|^{1/2}}{2^{k-n}} \Bigg \langle
\sum_{\substack{
            P \in \mathcal{D}_k(L) \\
            P \subset I^{+}}}
P_L^i \big ( \langle f \rangle _P -  \langle f \rangle _L \big ) +
\sum_{\substack{
            P \in \mathcal{D}_k(L) \\
            P \subset I^{-}}}
P_L^i \big ( \langle f \rangle _L -  \langle f \rangle _P \big ) , \Bigg. \\
& \hspace{6 cm} \Bigg.
\sum_{\substack{
            Q \in \mathcal{D}_k(L) \\
            Q \subset J^{+}}}
P_L^i \big( \langle g \rangle _Q -  \langle g \rangle _L \big ) +
\sum_{\substack{
            Q \in \mathcal{D}_k(L) \\
            Q \subset J^{-}}}
P_L^i \big ( \langle g \rangle _L -  \langle g \rangle _Q \big )
\Bigg \rangle_{\mathbb{C}^d}.
\end{align*}

We therefore have
\begin{align} \label{est:shift}
& \Big | \left \langle S_j f,g \right \rangle_{\ltw,\ltwi}  \Big|   \\ \nonumber
& \leq \sum_{L \in  \mathcal{L}_j} |L| \sum_{i = 1}^d 
\sum_{P,Q \in \mathcal{D}_k(L)} \bigg | \bigg \langle P_L^i \bigg (  \frac{\langle f \rangle _P -  \langle f \rangle _L}{2^k} \bigg ), P_L^i \bigg (\frac{\langle g \rangle _Q -  \langle g \rangle _L}{2^k}  \bigg )  \bigg \rangle_{\mathbb{C}^d}  \bigg|.  \nonumber
\end{align}

\section{The Bellman function}

We are now going to define the Bellman function associated to our problem. Let \(X >1\), fix a dyadic interval \(I_0\), and for \(\fd \in \mathbb{C}^d, \Fd \in \mathbb{R}, \ud \in \mathcal{M}_d(\mathbb{C}), \gd \in \mathbb{C}^d, \Gd \in \mathbb{R}, \vd \in \mathcal{M}_d(\mathbb{C})\) satisfying
\begin{equation}\label{domain}
\ud, \vd >0, I_d \leq \vd^{1/2} \ud \vd^{1/2} \leq X \cdot I_d, \|\vd^{-1/2} \fd\|_{\mathbb{C}^d}^2 \leq \Fd,\ \| \ud^{-1/2} \gd\|_{\mathbb{C}^d}^2 \leq \Gd ,
\end{equation}
define the function \(\fb_X=\fb_X^{I_0}: \mathbb{C}^d \times \mathbb{R} \times \mathcal{M}_d(\mathbb{C}) \times \mathbb{C}^d \times \mathbb{R} \times \mathcal{M}_d(\mathbb{C}),\) by
\begin{equation}   \label{bell}
\fb_X(\fd, \Fd, \ud, \gd, \Gd, \vd):=
 |I_0|^{-1} \sup \sum_{I \subseteq I_0} \left | \big \langle \sigma_I \big ( \langle f \rangle _{I^+} - \langle f \rangle_{I^-} \big ), \langle g \rangle _{I^+} - \langle g\rangle_{I^-}   \big \rangle_{\mathbb{C}^d} \right | \cdot |I|,
 \end{equation}
where the supremum is taken over all functions \(f,g :\mathbb{R} \to \mathbb{C}^d\) and matrix  \(A_2\) weights \(W\) on \(I_0\) such that
\begin{equation}\label{supdomain}
\langle f \rangle _{I_0}=\fd \in \mathbb{C}^d, \quad \big \langle \|W^{1/2} f\|^2_{\mathbb{C}^d} \big \rangle_{I_0} = \Fd \in \mathbb{R}, \quad \langle g \rangle _{I_0}=\gd \in \mathbb{C}^d, \quad \big \langle \|W^{-1/2} g\|^2_{\mathbb{C}^d} \big \rangle_{I_0} = \Gd \in \mathbb{R}, 
\end{equation}
\begin{equation}\label{weightdomain}
\sup_{\substack {
I \in \mathcal{D} \\
I \subset I_0} }
\|\langle W \rangle_I ^{1/2} \langle W^{-1} \rangle_I ^{1/2} \|^2 \leq X, \quad \langle W \rangle _{I_0} = \ud, \quad \langle W^{-1} \rangle_{I_0} = \vd,
\end{equation}
and all sequences of $d \times d$ matrices $\sigma = \{\sigma_I\}_{I \in \mathcal{D}}$ with $\|\sigma\|_{\infty, W} \le 1$.


The Bellman function \(\fb_X\) has the following properties:
\begin{enumerate}[(i)]
    \item (Domain) The domain \(\mathfrak{D}_X:=\mathrm{Dom}\, \fb_X\) is given by \eqref{domain}. This means that for every tuple \((\fd, \Fd, \ud, \gd, \Gd, \vd)\) that satisfies \eqref{domain}, there exist functions \(f, g\) and a matrix weight \(W\) such that \eqref{supdomain} holds, so the supremum is not \(-\infty\). Conversely, if the variables \(\fd, \Fd, \ud, \gd, \Gd, \vd\) are the corresponding averages of some functions \(f, g\) and \(W\), then they must satisfy condition \eqref{domain}. Since the set \(\{(\ud,\vd) \in \mathcal{M}_d(\mathbb{C}) \times \mathcal{M}_d(\mathbb{C}): \ud,\vd >0, I_d \leq \vd^{1/2} \ud \vd^{1/2} \leq X \cdot I_d\}\) is not convex, the domain \(\mathfrak{D}_X\) is not convex either. 
   \item (Range) \(0 \leq \fb_X(\fd, \Fd, \ud, \gd, \Gd, \vd) \leq 4  N(X) \Fd^{1/2} \Gd^{1/2}\) for all \((\fd, \Fd, \ud, \gd, \Gd, \vd) \in \mathfrak{D}_X.\)
    \item (Concavity condition) Consider all tuples \(A=(\fd, \Fd, \ud, \gd, \Gd, \vd), A_+=(\fd_+, \Fd_+, \ud_+, \gd_+, \Gd_+, \vd_+)\) and \(A_-=(\fd_-, \Fd_-, \ud_-, \gd_-, \Gd_-, \vd_-)\) in \(\mathfrak{D}_X\) such that \(A=(A_+ + A_-)/2\). 
    For all such tuples, we have the following concavity condition:
        \[\fb_X(A) \geq \frac{\fb_X(A_+)+\fb_X(A_-)}{2} + 
             \sup_{\| \tau\|_{\ud} \le 1 } \left| \left \langle  \tau (\fd_+ - \fd_- ), \gd_+ - \gd_-  \right \rangle_{\mathbb{C}^d}  \right| .\]
  \end{enumerate}
Here, the supremum is taken over all $d \times d$ matrices $\tau$ with $\|\tau\|_{\ud} := \|\ud^{1/2} \tau \ud^{-1/2}\| \le 1$.
\par
Let us now explain these properties of the function \(\fb_X\). For any matrix weight \(W\) and any interval \(I\) we have \(\langle W^{-1} \rangle_I ^{1/2} \langle W \rangle_I \langle W^{-1} \rangle_I ^{1/2} \geq I_d\), so \(\vd^{1/2} \ud \vd^{1/2} \geq I_d\). The inequality \(\vd^{1/2} \ud \vd^{1/2} \leq X \cdot I_d\) follows from the definition of the matrix \(A_2\) Muckenhoupt condition. Conversely, for any positive definite matrices \(\ud, \vd\) such that \(I_d \leq \vd^{1/2} \ud \vd^{1/2} \leq X \cdot I_d\), we can find a matrix weight \(W\) that satisfies \eqref{weightdomain}. To see this, we construct a matrix weight \(W\) that is constant on the children of  \(I_0\).\\
Given two matrices \(\ud\) and  \(\vd\) as above, we want to find two positive definite matrices, \(W_1\) and \(W_2\), such that 
\[\ud=\frac{1}{2}(W_1+W_2) \quad \mbox{and}  \quad \vd=\frac{1}{2}(W_1^{-1}+W_2^{-1}).\]
We have \(\ud = W_1 \vd W_2 = W_2 \vd W_1\), thus \(W_2^{-1} = \vd W_1 \ud^{-1} = \ud^{-1} W_1 \vd\). Let \(M:= \ud^{-1/2} W_1 \ud^{-1/2}\), \(\ N:= \ud^{-1/2} \vd^{-1} \ud^{-1/2}\), and notice that \(N \leq I_d\). Then the matrices \(M\) and \(N\) commute: 
\begin{align*}
N^{-1}M & = (\ud^{1/2} \vd \ud^{1/2}) (\ud^{-1/2} W_1 \ud^{-1/2}) = \ud^{1/2} (\vd W_1 \ud^{-1}) \ud^{1/2} \\
              & = \ud^{1/2} (\ud^{-1} W_1 \vd) \ud^{1/2} = (\ud^{-1/2} W_1 \ud^{-1/2})  (\ud^{1/2} \vd \ud^{1/2}) = MN^{-1}.
\end{align*}
Furthermore, \(\ud = \frac{1}{2}(W_1 + W_2) = \frac{1}{2}(W_1 + \ud W_1^{-1} \vd^{-1})\), so \(W_1 = \frac{1}{2}(W_1 \ud^{-1} W_1 + \vd^{-1})\). It follows that 
\[M = \frac{1}{2}(\ud^{-1/2} W_1 \ud^{-1} W_1 \ud^{-1/2} + \ud^{-1/2} \vd^{-1} \ud^{-1/2}) = \frac{1}{2}(M^2+N),\]
hence \(M\) satisfies the quadratic equation \((M^2 - 2M + I_d) - (I_d-N) = 0\). Choosing \(M = I_d + (I_d-N)^{1/2}\), we obtain
\[W_1 = \ud^{1/2} M \ud^{1/2} = \ud^{1/2} (I_d + (I_d-N)^{1/2}) \ud^{1/2},\]
and 
\[W_2 = 2\ud - W_1 = \ud^{1/2} (I_d - (I_d-N)^{1/2}) \ud^{1/2}.\]
It is clear that both \(W_1\) and \(W_2\) are positive definite matrices. We now set \(W:= W_1 \chi_{I_0^+} + W_2 \chi_{I_0^-}\) and notice that \(W\) satisfies the required properties \eqref{weightdomain}. 

The inequalities \(\|\vd^{-1/2} \fd\|_{\mathbb{C}^d}^2 \leq \Fd \) and \(\| \ud^{-1/2} \gd\|_{\mathbb{C}^d}^2 \leq \Gd\) follow from the Cauchy-Schwarz Inequality. To see this, choose a unit vector \(e \in \mathbb{C}^d\) such that \(\|\vd^{-1/2} \fd\|_{\mathbb{C}^d} = |\langle \vd^{-1/2} \fd ,e \rangle_{\mathbb{C}^d} |\). We then have 
\begin{align*}
 |\langle \vd^{-1/2} \fd ,e \rangle_{\mathbb{C}^d} | & = \big | \big \langle \vd^{-1/2} \langle f \rangle _{I_0} ,e  \big \rangle_{\mathbb{C}^d} \big | = \big | \big \langle \langle \vd^{-1/2}f \rangle _{I_0},e \big \rangle_{\mathbb{C}^d} \big | \\
  & = \bigg | \Big \langle \frac{1}{|I_0|} \int_{I_0} \vd^{-1/2} W^{-1/2}(t) W^{1/2}(t) f(t) \, \mathrm{d}t ,e \Big \rangle_{\mathbb{C}^d} \bigg |  \\
  & = \bigg | \frac{1}{|I_0|} \int_{I_0} \big \langle   W^{1/2}(t) f(t) ,W^{-1/2}(t) \vd^{-1/2}  e \big \rangle_{\mathbb{C}^d}  \mathrm{d}t \bigg |  \\
  & \leq \bigg ( \frac{1}{|I_0|} \int_{I_0}  \| W^{1/2}(t) f(t)\|_{\mathbb{C}^d}^2 \mathrm{d}t \bigg )^{1/2}  \bigg ( \frac{1}{|I_0|} \int_{I_0}  \| W^{-1/2}(t) \vd^{-1/2}  e\|_{\mathbb{C}^d}^2 \mathrm{d}t \bigg )^{1/2} \\
  & = \Fd^{1/2}  \bigg ( \frac{1}{|I_0|} \int_{I_0} \big  \langle W^{-1}(t) \vd^{-1/2}  e, \vd^{-1/2} e \big \rangle_{\mathbb{C}^d} \mathrm{d}t \bigg )^{1/2} \\
  & = \Fd^{1/2}  \Big \langle \frac{1}{|I_0|} \int_{I_0} W^{-1}(t) \vd^{-1/2}  e \, \mathrm{d}t , \vd^{-1/2} e \Big \rangle_{\mathbb{C}^d}  \\
  & = \Fd^{1/2} \big \langle \vd \vd^{-1/2}  e,   \vd^{-1/2} e \big \rangle_{\mathbb{C}^d} =   \Fd^{1/2},
\end{align*}
since all matrices involved are positive definite. The other inequality follows in the same way.

On the other hand, given a tuple \((\fd, \Fd, \ud, \gd, \Gd, \vd) \in \mathfrak{D}\) and a matrix weight \(W\) satisfying (\ref{weightdomain}),
we can always find two functions \(f, g\) satisfying \eqref{supdomain}. We first choose a function \(\phi : \mathbb{R} \to \mathbb{C}^d\) such that 
\[\int_{I_0} \phi(t)\, \mathrm{d}t =0, \quad  \int_{I_0} W(t) \phi(t)\, \mathrm{d}t =0, \quad  \frac{1}{|I_0|} \int_{I_0} \|W^{1/2}(t) \phi(t) \|^2_{\mathbb{C}^d}\, \mathrm{d}t = 1,\]
and then set \(f(t):= W^{-1}(t) \vd^{-1} \fd + (\Fd - \|\vd^{-1/2}\fd\|^2)^{1/2} \phi(t)\). It can be easily checked that this function has the required properties. A similar argument allows us to construct the function \(g\). 

Property (ii) follows from the definition of \(\fb_X\) and the inequality \eqref{linearisation}.

To prove the concavity condition, we consider three tuples \(A, A_+, A_- \in \mathfrak{D}_X\) such that \(A=(A_+ + A_-)/2\) and choose two functions \(f, g\) and a matrix weight \(W\) on \(I_0\) so that
\begin{equation}\label{avI0+-}
A_{\pm}= \Big ( \langle f \rangle _{I_0^{\pm}},\ \big \langle \|W^{1/2}f\|_{\mathbb{C}^d}^2 \big \rangle_{I_0^{\pm}},\ \langle W \rangle _{I_0^{\pm}},\ \langle g \rangle _{I_0^{\pm}},\ \big \langle \|W^{-1/2}g\|_{\mathbb{C}^d}^2 \big \rangle_{I_0^{\pm}},\ \langle W^{-1} \rangle _{I_0^{\pm}} \Big ).
\end{equation}
Then
\[A=\frac{A_+ + A_-}{2} = \Big ( \langle f \rangle _{I_0},\ \big \langle \|W^{1/2} f\|_{\mathbb{C}^d}^2 \big \rangle_{I_0},\ \langle W \rangle _{I_0},\ \langle g \rangle _{I_0},\ \big \langle \|W^{-1/2} g\|_{\mathbb{C}^d}^2 \big \rangle_{I_0}\ \langle W^{-1} \rangle _{I_0} \Big )\]
is the vector of corresponding averages over \(I_0\). The expression in the definition of \(\fb_X(\fd, \Fd, \ud, \gd, \Gd, \vd)\), before taking the supremum, can be split into the average of the corresponding expressions for \(\fb_X(\fd_+, \Fd_+, \ud_+, \gd_+, \Gd_+, \vd_+)\) and \(\fb_X(\fd_-, \Fd_-, \ud_-, \gd_-, \Gd_-, \vd_-)\), plus the term 
\[ \sup_{\sigma_{I_0}: \|\ud^{1/2}  \sigma_{I_0} \ud^{-1/2} \|  \le 1 } \big| \left \langle  \sigma_{I_0} (\fd_+ - \fd_- ), (\gd_+ - \gd_- ) \right \rangle_{\mathbb{C}^d} \big |. \]

 Taking now the supremum over all \(f, g\) and \(W\) that satisfy conditions \eqref{avI0+-} we conclude that
\[\frac{\fb_X(A_+)+\fb_X(A_-)}{2} + \sup_{\| \tau\|_{\ud} \le 1 } \big| \left \langle  \tau (\fd_+ - \fd_- ), (\gd_+ - \gd_- ) \right \rangle_{\mathbb{C}^d}  \big|  \leq \fb_X(A).\]


This inequality is true because the set of functions over which we are taking the supremum is smaller than the one corresponding to \(\fb_X(A)\), since we are excluding all those functions \(f, g\) and \(W\) whose averages on the children of \(I_0\) are not the prescribed values in \eqref{avI0+-}.
\begin{remark}
\normalfont
The concavity condition (iii) implies that the function \(\fb_X\) is midpoint concave, that is \(\fb_X \big(\frac{A_+ + A_-}{2}\big) \geq \frac{1}{2} \big(\fb_X(A_+) + \fb_X(A_-) \big)\), for all \(A_+, A_- \in \mathfrak{D}_X \) with \( \frac{A_+ + A_-}{2} \in \mathfrak{D}_X$.
 It is well-known that locally bounded below midpoint concave functions are actually concave (see e.g.~\cite{RoVa73}, Theorem C, p. 215). Therefore \(\fb_X\) is a concave function.
\end{remark}

We conclude this section with a result that allows us to overcome the non-convexity of the domain of the Bellman function.
\begin{lemma}\label{extdom}
Let \(A, A_+, A_- \in \mathfrak{D}_X\) such that \(A=(A_+ + A_-)/2\) . Then the line segment with endpoints \(A_+\) and \(A_-\) belongs to \(\mathfrak{D}_{4X}\).
\end{lemma}

\begin{proof}
We start by proving that the set $\mathfrak{D}_\infty$ given by the inequalities
\begin{equation*}
\ud, \vd >0, I_d \leq \vd^{1/2} \ud \vd^{1/2}, \|\vd^{-1/2} \fd\|_{\mathbb{C}^d}^2 \leq \Fd,\ \| \ud^{-1/2} \gd\|_{\mathbb{C}^d}^2 \leq \Gd 
\end{equation*}
is convex.

We first prove that the inequality \(\|\vd^{-1/2} \fd\|_{\mathbb{C}^d}^2 \leq \Fd\) is convex (the other inequality, \(\ \| \ud^{-1/2} \gd\|_{\mathbb{C}^d}^2 \leq \Gd\), follows in a similar way). It is enough to show that if \(\|\vd_1^{-1/2} \fd_1\|_{\mathbb{C}^d}^2 \leq \Fd_1\) and \(\|\vd_2^{-1/2} \fd_2\|_{\mathbb{C}^d}^2 \leq \Fd_2\), then 
\begin{equation}   \label{est:convex}
\bigg \| \left( \frac{1}{2}(\vd_1 + \vd_2)\right)^{-1/2} \frac{1}{2}(\fd_1 + \fd_2)\bigg \|_{\mathbb{C}^d}^2 \leq \frac{1}{2}(\Fd_1 + \Fd_2).
\end{equation}
We have 
\begin{align*}
\|(\vd_1 + \vd_2)^{-1/2} (\fd_1 + \fd_2)\|_{\mathbb{C}^d}^2 & = \left \langle (\vd_1 + \vd_2)^{-1}, (\fd_1 + \fd_2) \otimes (\fd_1 + \fd_2) \right \rangle _{HS}  \\
 & = \left \langle (\vd_1 + \vd_2)^{-1}, \fd_1 \otimes \fd_1 + (\fd_1 \otimes \fd_2 + \fd_2 \otimes \fd_1) + \fd_2 \otimes \fd_2 \right \rangle _{HS} \\
 & =: T_1 + T_2 + T_3,
 \end{align*}
 where \(\langle \cdot, \cdot \rangle_{HS}\) denotes the Hilbert-Schmidt (trace) inner product. \\
 Using the identities 
 \[(\vd_1 + \vd_2)^{-1} = \vd_1^{-1} - \vd_1^{-1} \vd_2 (\vd_1 + \vd_2)^{-1}  \quad \mbox{and}  \quad (\vd_1 + \vd_2)^{-1} = \vd_2^{-1} - \vd_2^{-1} \vd_1 (\vd_1 + \vd_2)^{-1}, \]
 we get that
 \[T_1 = \left \langle (\vd_1 + \vd_2)^{-1}, \fd_1 \otimes \fd_1 \right \rangle_{HS} = \|\vd_1^{-1/2} \fd_1\|_{\mathbb{C}^d}^2 - \left \langle \vd_1^{-1} \vd_2 (\vd_1 + \vd_2)^{-1}, \fd_1 \otimes \fd_1 \right \rangle_{HS},\]
 and 
 \[T_3 = \left \langle (\vd_1 + \vd_2)^{-1}, \fd_2 \otimes \fd_2 \right \rangle_{HS} = \|\vd_2^{-1/2} \fd_2\|_{\mathbb{C}^d}^2 - \left \langle \vd_2^{-1} \vd_1 (\vd_1 + \vd_2)^{-1}, \fd_2 \otimes \fd_2 \right \rangle_{HS}.\]

 Noting that $\vd_1^{-1} \vd_2 (\vd_1 + \vd_2)^{-1} = (\vd_1 + \vd_2)^{-1}  \vd_2 \vd_1^{-1} >0$ and writing $\tilde \fd_2 =\vd_1 \vd_2^{-1} \fd_2$, we find that
 \begin{align*}
T_1 + T_2 + T_3
 & \leq  -  \left \langle \vd_1^{-1} \vd_2 (\vd_1 + \vd_2)^{-1}, \fd_1 \otimes \fd_1 \right \rangle_{HS}  - \left \langle  \vd_1^{-1} \vd_2 (\vd_1 + \vd_2)^{-1}, \tilde \fd_2 \otimes \tilde \fd_2 \right \rangle_{HS} \\
& \qquad  + \left \langle  \vd_1^{-1} \vd_2 (\vd_1 + \vd_2)^{-1}, \fd_1 \otimes  \tilde \fd_2  \right \rangle_{HS}  
+   \left \langle  \vd_1^{-1} \vd_2 (\vd_1 + \vd_2)^{-1}, \tilde \fd_2 \otimes \fd_1  \right \rangle_{HS}   +  \Fd_1 + \Fd_2  \\
&=    -  \left \langle \vd_1^{-1} \vd_2 (\vd_1 + \vd_2)^{-1}, (\fd_1 -   \tilde \fd_2 ) \otimes( \fd_1 -  \tilde \fd_2 )   \right \rangle_{HS}+  \Fd_1 + \Fd_2
\leq \Fd_1 + \Fd_2. \\
 \end{align*}
 This concludes the proof of our claim.

We now check that the set \(C_0:=\{(\ud,\vd) \in \mathcal{M}_d(\mathbb{C}) \times \mathcal{M}_d(\mathbb{C}): \ud,\vd >0, I_d \leq \vd^{1/2} \ud \vd^{1/2}\}\) is convex. As before, it is enough to show that it is midpoint convex. \\
Let \( (\ud_1,\vd_1), (\ud_2,\vd_2) \in C_0\). We have to prove that 
\[I_d \leq  \left( \frac{\vd_1+\vd_2}{2}  \right)^{1/2}  \left( \frac{\ud_1+\ud_2}{2}  \right) \left( \frac{\vd_1+\vd_2}{2} \right)^{1/2},\]
which is equivalent to 
\[(\vd_1+\vd_2)(\ud_1+\ud_2)(\vd_1+\vd_2) \geq 4(\vd_1+\vd_2).\]
Since \( (\ud_1,\vd_1), (\ud_2,\vd_2) \in C_0\), we have \(\ud_1 \geq \vd_1^{-1}\) and \(\ud_2 \geq \vd_2^{-1}\), so
\begin {align*}
(\vd_1+\vd_2)(\ud_1+\ud_2)(\vd_1+\vd_2) & \geq (\vd_1+\vd_2)(\vd_1^{-1}+\vd_2^{-1})(\vd_1+\vd_2) \\
 & = 3 \vd_1 + 3 \vd_2 + \vd_1 \vd_2^{-1} \vd_1 + \vd_2 \vd_1^{-1} \vd_2 .
 \end{align*}
It is therefore enough to check that 
\[ \vd_1 \vd_2^{-1} \vd_1 + \vd_2 \vd_1^{-1} \vd_2  - \vd_1 - \vd_2 \geq 0,\]
which is the same as showing that 
\[\vd_1 ^{1/2}\vd_2^{-1} \vd_1^{1/2} + \vd_1 ^{-1/2}\vd_2 \vd_1^{-1/2} \vd_1 ^{-1/2}\vd_2 \vd_1^{-1/2} -I_d -  \vd_1 ^{-1/2}\vd_2 \vd_1^{-1/2} \geq 0.\]
Let \(T:=\vd_1 ^{1/2}\vd_2^{-1} \vd_1^{1/2} >0\). The previous inequality becomes \(T+T^{-2} - I_d -T^{-1} \geq 0\), which is equivalent to \(T^3 +I_d -T^2 - T \geq 0\). But \(T^3 +I_d -T^2 - T = (T-I_d)(T+I_d)(T-I_d),\) and this is a positive semidefinite matrix since \(T+I_d \geq 0\). This concludes the proof of the convexity of \(C_0\).  

To finish the proof of the lemma, we have to show that if \( (\ud,\vd), (\ud_+,\vd_+), (\ud_-,\vd_-)\) are in the set \(C_X:=\{(\ud,\vd) \in \mathcal{M}_d(\mathbb{C}) \times \mathcal{M}_d(\mathbb{C}): \ud,\vd >0,  \vd^{1/2} \ud \vd^{1/2} \leq X \cdot I_d\}\) and \((\ud,\vd) = \frac{1}{2}[(\ud_+,\vd_+) + (\ud_-,\vd_-)]\), then for all \(\theta \in [0,1]\), the points \((\ud_{\theta}, \vd_{\theta}) = (\theta \ud_+ + (1-\theta)\ud_-, \theta \vd_+ + (1-\theta)\vd_-)\) belong to the set \(C_{4X}\).

Since \(\theta \in [0,1]\), we have \(\theta \ud_+ \leq \ud_+\) and \( (1-\theta)\ud_- \leq \ud_-\), so \(\ud_{\theta} \leq \ud_+ + \ud_-=2 \ud\); we also have \(\vd_{\theta} \leq 2 \vd\). It is then sufficient to show that \(\vd_{\theta}^{1/2} (\ud_+ + \ud_-) \vd_{\theta}^{1/2} \leq 4X I_d\). But this is equivalent to \(\|\vd_{\theta}^{1/2} (\ud_+ + \ud_-) \vd_{\theta}^{1/2} \| \leq 4X\). All matrices that appear are positive definite, so \(\|\vd_{\theta}^{1/2} (\ud_+ + \ud_-) \vd_{\theta}^{1/2} \| = \| (\ud_+ + \ud_-) ^{1/2} \vd_{\theta} (\ud_+ + \ud_-) ^{1/2}\|\). Then again \( \| (\ud_+ + \ud_-) ^{1/2} \vd_{\theta} (\ud_+ + \ud_-) ^{1/2}\| \leq 4X\) if and only \((\ud_+ + \ud_-) ^{1/2} \vd_{\theta} (\ud_+ + \ud_-) ^{1/2} \leq 4XI_d\). 
We finally have 
\[(\ud_+ + \ud_-) ^{1/2} \vd_{\theta} (\ud_+ + \ud_-) ^{1/2} = 2 \ud^{1/2} \vd_{\theta} \ud^{1/2} \leq 4 \ud^{1/2} \vd \ud^{1/2}  \leq 4XI_d,\]
since \((\ud,\vd)\), and thus also \( (\vd,\ud)\), are in the set \(C_X\), so the proof of the lemma is complete.
\end{proof}

\section{The main estimate}

The following result is the main tool in the proof of Theorem \ref{mainshift}.
\begin{lemma}\label{mainlemma}
Let \(X>1\) and \(\fb_X\) be a function satisfying properties (i)-(iii) from Section 6. Fix \(k \geq 1\) and a dyadic interval \(I_0\). For all \(I \in \mathcal{D}_n(I_0),\ 0 \leq n \leq k,\) let the points \(A_I= (\fd_I, \Fd_I, \ud_I, \gd_I, \Gd_I, \vd_I) \in \mathfrak{D}_X=\mathrm{Dom}\, \fb_X\) be given. Assume that the points \(A_I\) satisfy the dyadic martingale dynamics, i.e. \(A=(A_{I^+}+A_{I^-})/2,\) where \(I^+\) and \(I^-\) are the children of \(I\).  Let \(B_{I_0} = \{e_{I_0}^1, e_{I_0}^2, \ldots, e_{I_0}^d\}\) be an orthonormal basis of eigenvectors of \( \ud_{I_0}\) and for $1 \le i \le d$, let  \(P_{I_0}^i\) be the orthogonal projection onto the span of \(e_{I_0}^i\).
For \(1 \leq i \leq d\) and \(K,L \in \mathcal{D}_k(I_0)\), we define the coefficients \(\lambda_{KL}^i\) by
\[\lambda_{KL}^i := \bigg \langle P_{I_0}^i \bigg ( \frac{\fd_K - \fd_{I_0}}{2^k} \bigg ), P_{I_0}^i \bigg ( \frac{\gd_L - \gd_{I_0}}{2^k} \bigg ) \bigg \rangle_{\mathbb{C}^d}  .\]

Then
\[\sum_{i = 1}^d \sum_{K,L \in \mathcal{D}_k({I_0})} |\lambda_{KL}^i| \leq c \cdot d \bigg( \fb_{X'}(A_{I_0}) - 2^{-k} \sum_{I \in \mathcal{D}_k(I_0)} \fb_{X'}(A_I) \bigg),\]
where \(c\) is a positive absolute constant and \(X'=\frac{100}{9}X\).
\end{lemma}

\begin{proof}

For \(1 \leq i \leq d\), we introduce the notation
\(\Lambda^i := \left (\lambda_{KL}^i \right )_{K,L \in \mathcal{D}_k(I_0)}.\) Assume for the moment that for each \(i\), we can find a sequence \(\{\alpha^i_I\}_{I \in \mathcal{D}_k(I_0)}\) such that \(|\alpha^i_I| \leq 1/4\) for all \(I \in \mathcal{D}_k(I_0), \sum_{I \in \mathcal{D}_k(I_0)} \alpha^i_I =0,\) and
\begin{equation}\label{sequence}
\bigg | \sum_{K,L \in \mathcal{D}_k({I_0})} \alpha^i_K \alpha^i_L \lambda_{KL}^i \bigg | \geq c \sum_{K,L \in \mathcal{D}_k({I_0})} |\lambda_{KL}^i|.
\end{equation}

For each \(i\), we define \(A_{I_0}^{i, \pm}=(\fd^{i, \pm}, \Fd^{i, \pm}, \ud^{i, \pm}, \gd^{i, \pm}, \Gd^{i, \pm}, \vd^{i, \pm})\) by
\begin{equation}   \label{eq:dyn}
A_{I_0}^{i, \pm}:=2^{-k} \sum_{I \in \mathcal{D}_k(I_0)} (1 \pm \alpha^i_I)A_I = A_{I_0} \pm 2^{-k} \sum_{I \in \mathcal{D}_k(I_0)} \alpha^i_I A_I,
\end{equation}
so \(A_{I_0}=(A_{I_0}^{i, +} + A_{I_0}^{i, -})/2\). 

The following notations and computations hold for every \(1 \leq i \leq d\), so we fix such an \(i\). For simplicity, we also drop the \(i\) superscript until further notice. 

For each \(I \in \mathcal{D}_k(I_0)\), let \(a_I^{\pm}:=1 \pm \alpha_I\) and note that \(3/4 \leq a_I^{\pm} \leq 5/4.\)

For \(I \in \mathcal{D}_n(I_0),\ 1 \leq n \leq k,\) let us define
\[A_I^{\pm}:= \left (
\sum_{\substack{
            J \in \mathcal{D}_k(I_0) \\
            J \subseteq I}}
a_J^{\pm}A_J \right ) \left (
\sum_{\substack{
            J \in \mathcal{D}_k(I_0) \\
            J \subseteq I}}
a_J^{\pm} \right )^{-1}.\]
If \(I \in \mathcal{D}_k(I_0)\) we have \(A_I^+=A_I^-=A_I\), where the \(A_I\)'s are the points from the statement of the lemma. The points \(A_I^{\pm}\) are in the convex hull of the set \(\{A_J: J \in \mathcal{D}_k(I_0), \ J \subseteq I\}\). To address the lack of convexity of $\mathfrak{D}_X$, we need an additional lemma:

\begin{lemma}  \label{lemm:convex}
 \(A_I^{\pm} \in \mathfrak{D}_{25/9X}\) for all \(I \in \mathcal{D}_n(I_0),\ 1 \leq n \leq k\).
\end{lemma}
\begin{proof} of Lemma \ref{lemm:convex}.
Since the points \(A_I^{\pm}\) are in the convex hull of the set \(\{A_J \in \mathfrak{D}_X \subset \mathfrak{D}_{25/9X}\}\), and among the conditions that define \(\mathfrak{D}_{25/9X}\) only the constraint \(\vd^{1/2} \ud \vd^{1/2} \leq \frac{25}{9}X \cdot I_d\) is not convex, we just have to check this condition. 

Let us consider the \(\ud\)-coordinate of the points \(A_I\). The maximal numerator is obtained when all coefficients \(a_J^{\pm}\) are equal to \(5/4\), and the minimal denominator is attained when \(a_J^{\pm}=3/4\) for all \(J \in \mathcal{D}_k(I_0), \ J \subseteq I\). This implies that \(\ud_I^{\pm} \leq \frac{5}{4} (\frac{3}{4})^{-1} \ud_I = \frac{5}{3} \ud_I\). Similarly, we also have \(\vd_I^{\pm} \leq \frac{5}{3} \vd_I\). Using elementary properties of positive definite matrices, it follows that 
\[(\vd_I^{\pm})^{1/2} \ud_I^{\pm} (\vd_I^{\pm})^{1/2} \leq \frac{5}{3} (\vd_I^{\pm})^{1/2} \ud_I (\vd_I^{\pm})^{1/2}\]
and 
\begin{align*}
\big \|(\vd_I^{\pm})^{1/2} \ud_I^{\pm} (\vd_I^{\pm})^{1/2} \big \| & \leq \frac{5}{3} \big \|(\vd_I^{\pm})^{1/2} \ud_I (\vd_I^{\pm})^{1/2} \big \|  = \frac{5}{3} \big \|\ud_I^{1/2} \vd_I^{\pm} \ud_I^{1/2} \big \| \\
& \leq \left( \frac{5}{3} \right)^2 \big \|\ud_I^{1/2} \vd_I \ud_I^{1/2} \big \| \leq \frac{25}{9}X, 
\end{align*}
hence \((\vd_I^{\pm})^{1/2} \ud_I^{\pm} (\vd_I^{\pm})^{1/2} \leq \frac{25}{9}X \cdot I_d\). This means that the points \(A_I^{\pm}\) belong to \(\mathfrak{D}_{25/9X}\).

Let \(\tilde{A}^\pm_{I}\) be the midpoints of the line segments with endpoints \(A_{I^+}^{\pm}\) and \(A_{I^-}^{\pm}\). We prove that \(\tilde{A}^\pm_{I} \in \mathfrak{D}_{25/9X}\).

As before, we have \(\ud_{I^{\pm}}^{\pm} \leq \frac{5}{3} \ud_{I^{\pm}}\) and \(\vd_{I^{\pm}}^{\pm} \leq \frac{5}{3} \vd_{I^{\pm}}\). Therefore,
\[\tilde{\ud}_I^{\pm} = \frac{\ud_{I^+}^{\pm} + \ud_{I^-}^{\pm}}{2} \leq \frac{5}{3} \frac{\ud_{I^+} +\ud_{I^-}}{2} = \frac{5}{3} \ud_I\]
and \(\tilde{\vd}_I^{\pm} \leq \frac{5}{3} \vd_I\). It follows that \((\tilde{\vd}_I^{\pm})^{1/2} \tilde{\ud}_I^{\pm} (\tilde{\vd}_I^{\pm})^{1/2} \leq \frac{25}{9}X \cdot I_d\), so the points \(\tilde{A}^\pm_{I}\) belong to \(\mathfrak{D}_{25/9X}\).

Applying Lemma \ref{extdom}, we conclude that the line segments with endpoints \(A_{I^+}^{\pm}\) and \(A_{I^-}^{\pm}\) are in \(\mathfrak{D}_{X'}\), where \(X'=4 \frac{25}{9} X = \frac{100}{9} X\).
This finishes the proof of Lemma \ref{lemm:convex}.
\end{proof} 

We continue with the proof of Lemma \ref{mainlemma}.
For \(I \in \mathcal{D}_n(I_0),\ 1 \leq n \leq k,\) we define
\[\theta_I^{\pm}:= \left (
\sum_{\substack{
            J \in \mathcal{D}_k(I_0) \\
            J \subseteq I}}
a_J^{\pm} \right ) \left (
\sum_{\substack{
            J \in \mathcal{D}_k(I_0) \\
            J \subseteq \tilde{I}}}
a_J^{\pm} \right )^{-1}.\]

\noindent
It is easy to see that \(3/10 \leq \theta_I^{\pm} \leq 5/6\) and
\begin{equation}\label{convexity}
\theta_{I^+}^{\pm} + \theta_{I^-}^{\pm} =1, \qquad A_I^{\pm}=\theta_{I^+}^{\pm} A_{I^+}^{\pm} + \theta_{I^-}^{\pm} A_{I^-}^{\pm}.
\end{equation}

\noindent
The last equality means that the point \(A_I^+\) is on the line segment with endpoints \(A_{I^+}^+\) and \(A_{I^-}^+\), and similarly for \(A_I^-\). \(\theta_{I^+}^{\pm}\) and \(\theta_{I^-}^{\pm}\) represent the probabilities of moving from the points \(A_I^{\pm}\) to \(A_{I^+}^{\pm}\) and \(A_{I^-}^{\pm}\), respectively.

Since by (\ref{eq:dyn})
\begin{align*}
\Big | \big \langle P_{I_0} \big (\fd^+ - \fd^- \big ), P_{I_0} \big (\gd^+ - \gd^- \big ) \big \rangle_{\mathbb{C}^d}  \Big | & = \bigg| \bigg \langle 2 \cdot 2^{-k} \sum_{I \in \mathcal{D}_k(I_0)} \alpha_I P_{I_0}(\fd_I), 2 \cdot 2^{-k} \sum_{I \in \mathcal{D}_k(I_0)} \alpha_I P_{I_0}(\gd_I) \bigg \rangle_{\mathbb{C}^d}  \bigg| \\
& = 4 \Big | \big \langle P_{I_0} \big (\fd^{\pm} - \fd_{I_0} \big) , P_{I_0} \big(\gd^{\pm} - \gd_{I_0} \big) \big \rangle_{\mathbb{C}^d} \Big |,
\end{align*}
we get by property (iii) of the Bellman function \(\fb_{X'}\)
\begin{multline}\label{bellmandif}
\big | \left \langle P_{I_0} \big (\fd^{\pm} - \fd_{I_0} \big) , P_{I_0} \big(\gd^{\pm} - \gd_{I_0} \big) \right \rangle_{\mathbb{C}^d} \big | 
   \leq \sup_{\| \tau\|_{\ud_{I_0} } \le 1 } \left| \left \langle  \tau (\fd^{\pm} - \fd_{I_0} ), (\gd^{\pm} - \gd_{I_0} ) \right \rangle_{\mathbb{C}^d}  \right| \\
   \leq \frac{1}{4} \bigg ( \fb_{X'}(A_{I_0}) - \frac{\fb_{X'}(A_{I_0}^+)+\fb_{X'}(A_{I_0}^-)}{2} \bigg ).
\end{multline}

From the concavity of the function \(\fb_{X'}\) and \eqref{convexity} it follows that
\[\fb_{X'}(A_I^{\pm}) \geq \theta_{I^+}^{\pm} \fb_{X'}(A_{I^+}^{\pm}) + \theta_{I^-}^{\pm} \fb_{X'}(A_{I^-}^{\pm}).\]

Applying now this inequality to \(I \in \mathcal{D}_n(I_0),\ 0 \leq n \leq k-1\), and taking into account that
\[\prod_{\substack{
            J \in \mathcal{D} \\
            I \subseteq J \subsetneq I_0 }}
\theta_J^{\pm} = a_I^{\pm} \bigg (\sum_{J \in \mathcal{D}_k(I_0)} a_J^{\pm} \bigg )^{-1} = 2^{-k} a_I^{\pm}\]
for all \(I \in \mathcal{D}_k(I_0),\) we obtain the estimate
\[\fb_{X'}(A_{I_0}^{\pm}) \geq 2^{-k} \sum_{I \in \mathcal{D}_k(I_0)} a_I^{\pm} \fb_{X'}(A_I).\]

Since \(a_I^+ + a_I^-=2\) when \(I \in \mathcal{D}_k(I_0),\) substituting the previous inequality in \eqref{bellmandif} gives
\begin{equation}\label{single_est}
\Big | \big \langle P_{I_0} \big (\fd^{\pm} - \fd_{I_0} \big) , P_{I_0} \big(\gd^{\pm} - \gd_{I_0} \big) \big \rangle_{\mathbb{C}^d}\Big | \leq \frac{1}{4} \bigg ( \fb_{X'}(A_{I_0}) - 2^{-k} \sum_{I \in \mathcal{D}_k(I_0)}  \fb_{X'}(A_I) \bigg ) .
\end{equation}

We are now ready to obtain the conclusion of the lemma. By \eqref{sequence} and the fact that \\ \(\sum_{I \in \mathcal{D}_k(I_0)} \alpha^i_I =0,\) we have the estimate
\begin{align}\label{est1}
c \sum_{i = 1}^d \sum_{K,L \in \mathcal{D}_k({I_0})} \left| \lambda_{KL}^i \right| & \leq \sum_{i = 1}^d \bigg | \sum_{K,L \in \mathcal{D}_k({I_0})} \alpha^i_K \alpha^i_L \lambda_{KL}^i \bigg | \\
& = \sum_{i = 1}^d  \bigg | \sum_{K,L \in \mathcal{D}_k({I_0})} \alpha^i_K \alpha^i_L \bigg \langle P_{I_0}^i \bigg ( \frac{\fd_K - \fd_{I_0}}{2^k} \bigg ), P_{I_0}^i \bigg ( \frac{\gd_L - \gd_{I_0}}{2^k} \bigg ) \bigg \rangle_{\mathbb{C}^d}  \bigg | \nonumber \\
& = \bigg | \sum_{i = 1}^d \bigg \langle 2^{-k} \sum_{I \in \mathcal{D}_k(I_0)} \alpha^i_I P_{I_0}^i (\fd_I - \fd_{I_0}),  2^{-k} \sum_{I \in \mathcal{D}_k(I_0)} \alpha^i_I P_{I_0}^i (\gd_I - \gd_{I_0}) \bigg \rangle_{\mathbb{C}^d} \bigg | \nonumber \\
& = \sum_{i = 1}^d \Big | \big \langle P_{I_0}^i \big( \fd^{i, \pm} - \fd_{I_0} \big) , P_{I_0}^i \big(\gd^{i, \pm} - \gd_{I_0} \big) \big \rangle_{\mathbb{C}^d} \Big | \nonumber \\
& \leq \frac{d}{4} \bigg ( \fb_{X'}(A_{I_0}) - 2^{-k} \sum_{I \in \mathcal{D}_k(I_0)}  \fb_{X'}(A_I) \bigg ), \nonumber
\end{align}
where the last inequality follows from \eqref{single_est}. This completes the proof of Lemma \ref{mainlemma} under the assumption \eqref{sequence}. 

For each \(1 \leq i \leq d\), the matrix \(\Lambda^i\) has complex rank 1. Dropping again the \(i\) superscript, there exist \(m = (m_K), n = (n_L) \in \mathbb{C}^{2^k}\) such that  \(\lambda_{KL} = m_K n_L = (m^1_K + i m^2_K) (n^1_L + i n^2_L)\), for every \(K,L \in \mathcal{D}_k(I_0)\).  We then have 
\[ \sum_{K,L \in \mathcal{D}_k(I_0)} |\lambda_{KL}| \leq \sum_{K,L \in \mathcal{D}_k(I_0)} \big( |m^1_K n^1_L| + |m^1_K n^2_L| + |m^2_K n^1_L| + |m^2_K n^2_L| \big).\]
Without loss of generality, we may assume that \(\sum_{K,L \in \mathcal{D}_k(I_0)} |m^1_K n^1_L|\) is the maximum of the four sums in the above right hand side. By an application of K. Ball's ``multiple Hahn-Banach Theorem'' (\cite{Ba91}, Theorem 7), or alternatively an elementary functional analysis argument (see \cite{Tr11}, Theorem 6.2 and Lemma 6.3), we can find a real-valued sequence \(\{\alpha_I\}_{I \in \mathcal{D}_k(I_0)}\) such that \(|\alpha_I| \leq 1/4\) for all \(I \in \mathcal{D}_k(I_0), \sum_{I \in \mathcal{D}_k(I_0)} \alpha_I =0,\) and 
\[\bigg| \sum_{K \in \mathcal{D}_k(I_0)} \alpha_K m^1_K \bigg| \geq \frac{1}{16} \sum_{K \in \mathcal{D}_k(I_0)} |m^1_K| , \qquad \bigg| \sum_{L \in \mathcal{D}_k(I_0)} \alpha_L n^1_L \bigg| \geq \frac{1}{16} \sum_{L \in \mathcal{D}_k(I_0)} |n^1_L|.\]

It follows that 
\begin{align*}
\sum_{K,L \in \mathcal{D}_k(I_0)} |\lambda_{KL}| & \leq 4 \sum_{K,L \in \mathcal{D}_k(I_0)} |m^1_K n^1_L| \leq 4^5 \bigg| \sum_{K \in \mathcal{D}_k(I_0)} \alpha_K m^1_K \bigg| \cdot \bigg| \sum_{L \in \mathcal{D}_k(I_0)} \alpha_L n^1_L \bigg| \nonumber \\
& \leq 4^5 \bigg| \sum_{K \in \mathcal{D}_k(I_0)} \alpha_K (m^1_K + i m^2_K) \bigg| \cdot \bigg| \sum_{L \in \mathcal{D}_k(I_0)} \alpha_L (n^1_L + i n^2_L) \bigg| \nonumber \\
& = 4^5 \bigg| \sum_{K \in \mathcal{D}_k(I_0)} \alpha_K m_K \bigg| \cdot \bigg| \sum_{L \in \mathcal{D}_k(I_0)} \alpha_L n_L \bigg| =  4^5 \bigg| \sum_{K, L \in \mathcal{D}_k(I_0)} \alpha_K \alpha_L \lambda_{KL} \bigg|, \nonumber
\end{align*}
which is what we wanted to show. Therefore, the proof of the lemma is complete, with $c=4^5$.
\end{proof}

\section{Conclusion of the proof of Theorem \ref{mainshift}}
We are now ready to finish the proof of Theorem \ref{mainshift}.

Recall that for all slices \(S_j\) of \(S\) we have
\[ \Big | \left \langle S_j f,g \right \rangle_{\ltw,\ltwi} \Big|  \leq \sum_{L \in  \mathcal{L}_j} |L| \sum_{i = 1}^d 
\sum_{P,Q \in \mathcal{D}_k(L)} \bigg | \bigg \langle P_L^i \bigg (  \frac{\langle f \rangle _P -  \langle f \rangle _L}{2^k} \bigg ), P_L^i \bigg (\frac{\langle g \rangle _Q -  \langle g \rangle _L}{2^k}  \bigg )  \bigg \rangle_{\mathbb{C}^d}  \bigg|. \]

Let \(X:= [W]_{A_2}\); fix \(0 \leq j \leq k-1 \) and for all \(I \in \mathcal{L}_j\) define
\[A_I:=\Big (\langle f \rangle _I, \big \langle \|W^{1/2}f\|^2_{\mathbb{C}^d}  \big \rangle_I, \langle W \rangle _I, \langle g \rangle _I, \big \langle \|W^{-1/2}g\|^2_{\mathbb{C}^d} \big \rangle_I, \langle W^{-1} \rangle _I \Big).\]
Notice that all these points are in \(\mathrm{Dom}\, \fb_X=\mathfrak{D}_X\). Lemma \ref{mainlemma} says that
\begin{align*}
& |L| \sum_{i = 1}^d 
\sum_{P,Q \in \mathcal{D}_k(L)} \bigg | \bigg \langle P_L^i \bigg (  \frac{\langle f \rangle _P -  \langle f \rangle _L}{2^k} \bigg ), P_L^i \bigg (\frac{\langle g \rangle _Q -  \langle g \rangle _L}{2^k}  \bigg )  \bigg \rangle_{\mathbb{C}^d}  \bigg| \\
& \qquad \qquad \leq c \cdot d \bigg( |L| \fb_{X'}(A_L) - \sum_{I \in \mathcal{D}_k(L)} |I| \fb_{X'}(A_I) \bigg ),
\end{align*}
for all \(L \in \mathcal{L}_j.\) We write this estimate for each \(I \in \mathcal{D}_k(L)\) and then iterate the procedure \(\ell\) times to obtain
\begin{align*}
&  \sum_{\substack{
                    I \in \mathcal{L}_j \\
                    I \subseteq L\\
                    |I| > 2^{-k \ell}|L| }}
|I| \sum_{i = 1}^d \sum_{P,Q \in \mathcal{D}_k(I)} \bigg | \bigg \langle  P_I^i \bigg ( \frac{\langle f \rangle _P -  \langle f \rangle _I}{2^k} \bigg), P_I^i \bigg ( \frac{\langle g \rangle _Q -  \langle g \rangle _I}{2^k} \bigg)    \bigg \rangle_{\mathbb{C}^d} \bigg|  \\
& \qquad \qquad \qquad \leq c \cdot d  \bigg( |L| \fb_{X'}(A_L) - \sum_{I \in \mathcal{D}_{k \ell}(L)} |I| \fb_{X'}(A_I) \bigg ) \\
& \qquad \qquad \qquad \leq c \cdot d   |L| \fb_{X'}(A_L)  \\
& \qquad \qquad \qquad \leq c \cdot d N(X') |L| \big \langle \|W^{1/2}f\|^2_{\mathbb{C}^d}  \big \rangle_L^{1/2} \big \langle \|W^{-1/2}g\|^2_{\mathbb{C}^d}  \big \rangle_L^{1/2} \\
& \qquad \qquad \qquad \leq c \cdot d N(X') \|f \chi_L\|_{\ltw}  \|g \chi_L\|_{\ltwi},
\end{align*}
where the second inequality follows from property (ii) of the Bellman function.
\par Letting \(\ell \to \infty\), we have
\begin{align*}
&  \sum_{\substack{
                    I \in \mathcal{L}_j \\
                    I \subseteq L }}
|I| \sum_{i = 1}^d \sum_{P,Q \in \mathcal{D}_k(I)} \bigg | \bigg \langle  P_I^i \bigg ( \frac{\langle f \rangle _P -  \langle f \rangle _I}{2^k} \bigg), P_I^i \bigg ( \frac{\langle g \rangle _Q -  \langle g \rangle _I}{2^k} \bigg)    \bigg \rangle_{\mathbb{C}^d} \bigg|  \\
& \qquad \qquad \qquad \leq c \cdot d N(X') \|f \chi_L\|_{\ltw}  \|g \chi_L\|_{\ltwi}.
\end{align*}

We now cover the real line with intervals \(L \in \mathcal{L}_j\) of length \(2^M\) and apply the last inequality to each \(L\) to obtain that
\begin{align*}
&  \sum_{\substack{
                    I \in \mathcal{L}_j \\
                    |I| \leq 2^M }}
|I| \sum_{i = 1}^d \sum_{P,Q \in \mathcal{D}_k(I)} \bigg | \bigg \langle  P_I^i \bigg ( \frac{\langle f \rangle _P -  \langle f \rangle _I}{2^k} \bigg), P_I^i \bigg ( \frac{\langle g \rangle _Q -  \langle g \rangle _I}{2^k} \bigg)    \bigg \rangle_{\mathbb{C}^d} \bigg|  \\
& \qquad \qquad \qquad \leq c \cdot d N(X) \|f\|_{\ltw}  \|g\|_{\ltwi}.
\end{align*}

For \(M \to \infty\), we get that the norm of \(S_j\) is bounded by \(c \cdot d N(X)\). Since \(S\) was decomposed into \(k\) slices, it follows that the operator norm of \(S\) is bounded by \(c \cdot k d N([W]_{A_2^d})\), and therefore the proof of Theorem \ref{mainshift} is complete.
\qed

Using the bound for matrix-weighted dyadic martingale transforms proved in \cite{BiPeWi14} and the bound for matrix-weighted paraproducts in \cite{IHP}, page 7, together with Hyt\"onen's representation theorem in \cite{Hy11}, we obtain
the following consequence of Theorem \ref{mainthm}:
\begin{theorem}
Let \(W\) be a \(d \times d\) matrix \(A_2\) weight on \(\mathbb{R}^p\). 
Let \(K\) be a standard kernel and \(T\) be a Calder\'{o}n-Zygmund operator on \(\mathbb{R}^p\) associated to \(K\). Suppose that \(T\) satisfies the weak boundedness property \(|\langle T \chi_Q, \chi_Q \rangle |
\leq C|Q|\) for all cubes \(Q\). Then
\[\|T\|_{\ltw \to \ltw} \leq C \cdot 2^{3(p-1)} p ([W]_{A_2})^{3/2}  (2(p-1) +\log ([W]_{A_2}) )   ,\]
where \(C\) depends only on the constants in the standard estimates and the weak boundedness property, and the dimension $d$.
\end{theorem}

\begin{remark} \label{rmk1}
\normalfont
Obviously, we have not used the full power of the Bellman function here - the supremum in the Bellman function is taken over all $\tau$ with $\|\mathbf{U}^{1/2} \tau \mathbf{U}^{-1/2}\|\le1$, while we have only used the projections on the eigenspaces of $\mathbf{U}$. The setup actually allows to treat matrix-valued kernels as well, using the recent representation theorem for Calder\'{o}n-Zygmund operators with operator-valued kernels in \cite{HaHy14}, which again gives a decomposition into dyadic shifts. However, in the matrix-weighted setting, one needs to adapt the decay conditions on the Calder\'{o}n-Zygmund operator to the matrix weight $W$  (see \cite{Is15}, page 3). This approach is the subject of the paper \cite{PoSt17}.
\end{remark}

\begin{remark} \label{rmk2} 
\normalfont
Following Remark \ref{rmk1}, we could also have used a smaller version of the function $N$ by choosing a smaller class of martingale transforms for our proof, namely for example
$$
N_1(X) =  \sup\ \|T_{\sigma}\|_{\ltw \to \ltw},
$$
where the supremum is taken over all  $ d \times d$  matrix $A_2$ weights $W$ with $ [W]_{A_2} \le X$ and all sequences of $d \times d$ matrices
 $\sigma = \{\sigma_I\}_{I \in \mathcal{D}}$  with $\|\sigma_I\| \le 1$ and $ \sigma_I$ commuting with $\langle W \rangle_I $ for all $ I \in \mathcal{D}$.
 
One can then define the Bellman function with the projections $P_{I}^i$ from Lemma \ref{mainlemma} instead of the $\tau$, running exactly the same proof.
The reason we used the more general class of martingale transforms is that for both classes of $\sigma$'s, we have the pointwise estimate
$$
S_W(T_\sigma f)(t) \le S_W f(t),
$$
where $S_W$ is the matrix-weighted square function (see \cite{BiPeWi14}, \cite{PePo02}). Our expectation here was that the norm growth of the matrix-weighted square function controls the norm growth of the matrix-weighted martingale transforms in terms of $[W]_{A_2}$, and that both bounds are linear in $[W]_{A_2}$. This would, by Theorem \ref{mainthm},  
imply the linear bound in $[W]_{A_2}$ for general Calder\'{o}n-Zygmund operators with cancellation. Indeed, the linear bound of the matrix-weighted square function has  been proved after this paper was 
refereed \cite{hpv}. The linear bound for martingale transforms remains currently open. An account on possible strategies and some of the obstacles can be found in Section 6 of \cite{BiPeWi14}.
\end{remark}

\section{More about Calderón-Zygmund operators with even kernel}
One of the key aspects of the definition of the martingale transform operator in Section 2.2 is that the matrices \(\sigma_I\) interact well with the weight \(W\) (for the proof of our main result, we have essentially used the special case where 
the \(\sigma_I\)'s are diagonal in some basis). 

In the scalar-valued case, the definition of the martingale transform is simpler. More precisely, for a real sequence \(\sigma=\{\sigma_I\}_{I \in \dd},\ \sigma_I= \pm 1,\) we define the martingale transform operator \(\tilde{T}_{\sigma}\) by
\[\tilde{T}_\sigma f = \sd \sigma_I  \langle f,h_I \rangle h_I.\]
Allowing this operator to act on vector-valued functions, we can prove a similar result to Theorem \ref{mainthm}, but this time, the bound will only apply to Calder\'{o}n-Zygmund operator
with even kernels and sufficient smoothness of the kernel. For this, we define the function \(\tilde{N}:[1, \infty) \to [1, \infty)\) by
\[ \tilde{N}(X)=  \sup \|\tilde{T}_{\sigma}\|_{\ltw \to \ltw}, \]
where the supremum is taken over all real sequences \(\sigma\) as above and all \(d \times d\) matrix \(A_2^d\) weights \(W\) on \(\mathbb{R}\) with \([W]_{A_2^d} \leq X\).

\begin{theorem}\label{mainthm_even}
Let \(W\) be a \(d \times d\) matrix \(A_2\) weight on \(\mathbb{R}^p\). Let \(\tilde{K}\) be an even standard kernel with smoothness \(\delta > 1/2\) and \(T\) be a Calder\'{o}n-Zygmund operator on \(\mathbb{R}^p\) associated to \(\tilde{K}\). Suppose that \(T\) satisfies the weak boundedness property \(|\langle T \chi_Q, \chi_Q \rangle | \leq C|Q|\) for all cubes \(Q\), and the vanishing paraproduct conditions \(T(1)=T^*(1)=0\). Then
\[\|T\|_{\ltw \to \ltw} \leq C \cdot  pd \tilde{N}(2^{2(p-1)} [W]_{A_2}) \leq C_p \cdot  d \tilde{N}([W]_{A_2}),\]
where \(C\) depends only on the constants in the standard estimates and the weak boundedness property, while \(C_p\) depends on \(C\) and \(p\).
\end{theorem}

As before, the proof of this result follows from a corresponding inequality for self-adjoint Haar shift operators. More precisely, we will show the estimate 
\begin{equation}
    \|S^{mn}\|_{\ltw \to \ltw} \lesssim (\max\{m,n\}+1) 2^{\max\{m,n\}/2} \tilde{N} ([W]_{A_2})
\end{equation}
     for all dyadic Haar shifts \(S^{mn}\) of parameters \((m,n)\), which ensures the convergence of the series (since \(\delta > 1/2\)) in the representation theorem. This is the content of the following theorem.

\begin{theorem}\label{mainshift_even}
Let \(S\) be a self-adjoint dyadic Haar shift on \(\mathbb{R}\) of complexity \(k \geq 1\) and \(W\) be a matrix \(A_2^d\) weight. Then
\[\|S\|_{\ltw \to \ltw} \leq c \cdot k 2^{k/2} \tilde{N} ([W]_{A_2^d}),\]
where c is an absolute, positive constant.
\end{theorem}

The reduction of the proof follows almost like in Section 3, except that the orthogonal projection operators \(P_I^i\) don't appear. Since the dyadic Haar shift \(S\) is self-adjoint, we obtain the following estimate:
\begin{align*}
& 2 \Big | \left \langle S_j f,g \right \rangle_{\ltw,\ltwi}  \Big| = \Big | \left \langle S_j f,g \right \rangle_{\ltw,\ltwi}  + \left \langle f,S_j g \right \rangle_{\ltw,\ltwi} \Big |  \\
& \leq \sum_{L \in  \mathcal{L}_j} |L|
\sum_{P,Q \in \mathcal{D}_k(L)} \bigg | \bigg \langle  \frac{\langle f \rangle _P -  \langle f \rangle _L}{2^k}, \frac{\langle g \rangle _Q -  \langle g \rangle _L}{2^k}    \bigg \rangle_{\mathbb{C}^d}    +  \bigg \langle  \frac{\langle f \rangle _Q -  \langle f \rangle _L}{2^k}, \frac{\langle g \rangle _P -  \langle g \rangle _L}{2^k}    \bigg \rangle_{\mathbb{C}^d}   \bigg|.
\end{align*}

With the same notations as in Section 4, the Bellman function \(\fb_X\) is defined by 
\[\fb_X(\fd, \Fd, \ud, \gd, \Gd, \vd):= |I_0|^{-1} \sup \sum_{I \subseteq I_0} \left | \big \langle \langle f \rangle _{I^+} - \langle f \rangle_{I^-}, \langle g \rangle _{I^+} - \langle g\rangle_{I^-}  \big \rangle_{\mathbb{C}^d} \right | \cdot |I|.\]
The only differences between the properties of this function and those of the old Bellman function (\ref{bell}) are the replacement of \(N(X)\) by \(\tilde{N}(X)\) in (ii) and the absence of the operators \(P_I^i\) in (iii).

The proof of Theorem \ref{mainshift_even} is based on the following result, which is a similar version of Lemma \ref{mainshift}.
\begin{lemma}\label{mainlemma_even}
Let \(X>1\) and \(\fb_X\) be a function satisfying properties (i)-(iii) from Section 6. Fix \(k \geq 1\) and a dyadic interval \(I_0\). For all \(I \in \mathcal{D}_n(I_0),\ 0 \leq n \leq k,\) let the points \(A_I= (\fd_I, \Fd_I, \ud_I, \gd_I, \Gd_I, \vd_I) \in \mathfrak{D}_X=\mathrm{Dom}\, \fb_X\) be given. Assume that the points \(A_I\) satisfy the dyadic martingale dynamics, i.e. \(A=(A_{I^+}+A_{I^-})/2,\) where \(I^+\) and \(I^-\) are the children of \(I\). For \(K,L \in \mathcal{D}_k(I_0)\), we define the coefficients \(\lambda_{KL}\) by
\[\lambda_{KL}:= \bigg \langle  \frac{\fd_K - \fd_{I_0}}{2^k}, \frac{\gd_L - \gd_{I_0}}{2^k} \bigg \rangle_{\mathbb{C}^d}  + \bigg \langle  \frac{\fd_L - \fd_{I_0}}{2^k}, \frac{\gd_K - \gd_{I_0}}{2^k} \bigg \rangle_{\mathbb{C}^d}.\]

Then
\[\sum_{K,L \in \mathcal{D}_k({I_0})} |\lambda_{KL}| \leq c \cdot 2^{k/2} \bigg( \fb_{X'}(A_{I_0}) - 2^{-k} \sum_{I \in \mathcal{D}_k(I_0)} \fb_{X'}(A_I) \bigg),\]
where \(c\) is a positive absolute constant and \(X'=\frac{100}{9}X\).
\end{lemma}

The only difference between the proof of this result and that of Lemma \ref{mainshift} is the way to obtain the existence of the real sequence  
\( \{\alpha_I\}_{I \in \mathcal{D}_k(I_0)}\) such that  \(|\alpha_I| \leq 1/4\) for all  \(I \in \mathcal{D}_k(I_0),  \\ \sum_{I \in \mathcal{D}_k(I_0)} \alpha_I =0,\) and
\begin{equation}\label{sequence_even}
\bigg | \sum_{K,L \in \mathcal{D}_k({I_0})} \alpha_K \alpha_L \lambda_{KL} \bigg | \geq c \cdot 2^{-k/2} \sum_{K,L \in \mathcal{D}_k({I_0})} |\lambda_{KL}|.
\end{equation}
We will use again the notation \(\Lambda := \left (\lambda_{KL} \right )_{K,L \in \mathcal{D}_k(I_0)}.\) Let us now define
\[\|\Lambda\|_1 := \sup_{\alpha} \bigg | \sum_{K,L \in \mathcal{D}_k({I_0})} \alpha_K \alpha_L \lambda_{KL} \bigg|,\]
where the supremum is taken over all real sequences \(\alpha=\{\alpha_I\}_{I \in \mathcal{D}_k(I_0)}\) with \(\|\alpha\|_{\infty} \leq 1/4\) and \(\sum_{I \in \mathcal{D}_k(I_0)} \alpha_I =0.\) Since we are in a finite-dimensional space, we can find a sequence \(\alpha\) with \(|\alpha_I| \leq 1/4,\ I \in \mathcal{D}_k(I_0),\) and  \(\sum_{I \in \mathcal{D}_k(I_0)} \alpha_I =0,\) such that
\begin{equation}\label{norm1}
\bigg | \sum_{K,L \in \mathcal{D}_k({I_0})} \alpha_K \alpha_L \lambda_{KL} \bigg| = \|\Lambda\|_1.
\end{equation}

Using the symmetry of \(\Lambda\) and the fact that its row and column sums are all zero, it is easy to see that \(\|\Lambda\|_1\) is equivalent to
\[\|\Lambda\|_2 := \sup_{\substack{
                                   \|\alpha\|_{\infty} \leq 1 \\
                                   \|\beta\|_{\infty} \leq 1}}
\bigg| \sum_{K,L \in \mathcal{D}_k({I_0})} \alpha_K \beta_L \lambda_{KL} \bigg|,\]
where we take the supremum over all real sequences \( \alpha=\{\alpha_I\}_{I \in \mathcal{D}_k(I_0)}\) and \(\beta=\{\beta_I\}_{I \in \mathcal{D}_k(I_0)}.\) More precisely, we have \(64 \|\Lambda\|_1 \leq \|\Lambda\|_2 \leq 192 \|\Lambda\|_1.\) Since we may assume that \(\Lambda\) is not the zero matrix (otherwise the lemma becomes trivially true), \(\|\Lambda\|_1,\) and hence \(\|\Lambda\|_2\), are not \(0.\)

We also need the notion of Schur multiplier. If \(A=(a_{ij}) \in \mathcal{M}_n(\mathbb{R})\), the Schur multiplier is the bounded operator \(S_A:\mathcal{M}_n(\mathbb{R}) \to \mathcal{M}_n(\mathbb{R})\) that acts on a matrix \(M=(m_{ij})\) by Schur multiplication: \(S_A(M)=(a_{ij}m_{ij})\). The Schur multiplier norm is
\[\|A\|_m:=\sup_{M \in \mathcal{M}_n(\mathbb{C})} \frac{\|S_A(M)\|_{op}}{\|M\|_{op}},\]
where \(\|M\|_{op}\) is the operator norm of the matrix \(M\) on \(\ell^2(\{1,2,\ldots,n\})\). If \(A=(a_{ij}) \in \mathcal{M}_n(\mathbb{R})\) is of the form \(a_{ij}=s_i t_j\), then \(A\) is called a rank one Schur multiplier. It is easy to see that if \(A\) is a rank one Schur multiplier, then \(\|A\|_m \leq \|(s_i)_{i=1}^n\|_{\infty} \|(t_j)_{j=1}^n\|_{\infty} \). A classical result due to A. Grothendieck says that the converse is essentially true (up to a constant called Grothendieck constant).

\begin{theorem} [\protect{\cite[Theorem 1.2]{DaDo07}, \cite[Theorem 3.2]{Pi12}}] \label{grot}
The closure of the convex hull of the rank one Schur multipliers of norm one in the topology of pointwise convergence contains the ball of all Schur multipliers of norm at most \(K_G^{-1}\), where \(K_G\) is a universal constant.
\end{theorem}

If \(\alpha\) and \(\beta\) are two real sequences as above, the matrix \(\Phi=(\Phi_{KL})=(\alpha_K \beta_L)\) is a rank one Schur multiplier of norm \(\|\Phi\|_m\) at most \(1\). The inequality \(\|\Lambda\|_2 \leq 192 \|\Lambda\|_1\) can thus be rewritten as
\[\sup \big |\langle \Phi, \Lambda \rangle_{HS} \big | \leq 192 \|\Lambda\|_1,\]
where $\langle \cdot, \cdot \rangle_{HS}$ denotes inner product on the Hilbert-Schmidt class and the supremum is taken over all rank one Schur multipliers \(\Phi\) of norm at most \(1\).

Using Theorem \ref{grot}, we obtain that
\[\sup \big |\langle M, \Lambda \rangle_{HS} \big | \leq 192 K_G \|\Lambda\|_1,\]
where the supremum is now taken over all Schur multipliers \(M\) of norm at most \(1\), and \(K_G\) is the (real) Grothendieck constant.

By choosing either the real or the imaginary part of the matrix \(\Lambda\) (the one with greater \(\ell^1\)-norm), we have \( \sum_{K,L \in \mathcal{D}_k({I_0})} |\lambda_{KL}| \leq 2 \sup \big | \langle M, \Lambda \rangle_{HS} \big |,\) where the supremum is taken over all matrices \(M \in \mathcal{M}_{2^k}(\mathbb{R})\) with entries \(\pm 1\). For such a matrix \(M\) we have \(\|M\|_m \leq 2^{k/2}\), see \cite{DaDo07}, Lemma 2.5. Putting everything together, we get the inequality
\begin{equation}\label{sumbdd}
\sum_{K,L \in \mathcal{D}_k({I_0})} |\lambda_{KL}| \leq 384 K_G 2^{k/2} \|\Lambda\|_1.
\end{equation}
Using \eqref{norm1} and \eqref{sumbdd}, it follows that there exists a real sequence \(\alpha\) with \(|\alpha_I| \leq 1/4\) for all \(I \in \mathcal{D}_k(I_0),\) and  \(\sum_{I \in \mathcal{D}_k(I_0)} \alpha_I =0,\) such that
\[\sum_{K,L \in \mathcal{D}_k(I_0)} |\lambda_{KL}| \leq 384 K_G 2^{k/2} \|\Lambda\|_1 = 384 K_G 2^{k/2} \bigg | \sum_{K,L \in \mathcal{D}_k({I_0})} \alpha_K \alpha_L \lambda_{KL} \bigg|,\]
which is what we wanted to show. Since the other arguments are the same as in Lemma \ref{mainshift}, this completes the proof of Lemma \ref{mainlemma_even}.

The inequality in Theorem \ref{mainshift_even} is now obtained as in Section 6.

\section{A matrix version of the weighted Carleson Embedding Theorem}

In this section we will prove a version of the matrix-weighted Carleson Embedding Theorem. In the weighted setting, contrary to the unweighted case, the scalar-valued Carleson Embedding Theorem cannot be used to obtain the matrix version of the theorem. Here is the main result of this section.

\begin{theorem}[Matrix Carleson Embedding Theorem]\label{matrix_CET}
Let \(W\) be a \(d \times d\) matrix weight, and \(\{A_I\}_{I \in \dd}\) be a sequence of \( d \times d\) positive definite matrices. Then for $0 < t\le 1$,
\[
t\sd \big \langle (I_d +  t \langle W \rangle_I ^{-1} \widetilde{\md}_I)^{-1} A_I  (I_d +  t \widetilde{\md}_I  \langle W \rangle_I^{-1})^{-1} \langle W^{1/2} f \rangle_I, \langle W^{1/2} f \rangle_I \big \rangle_{\mathbb{C}^d} \leq 8 \|f\|^2_{\ltr}\]
if 
\[\frac{1}{|I|} \sum_{J \subseteq I} \langle W \rangle_J  A_J \langle W \rangle_J  \leq \langle W \rangle_I, \quad \text{for all }  I \in \dd,\]
where \(\widetilde{\md}_J = \frac{1}{|J|} \sum_{K \subsetneq J} \langle W \rangle_K A_K \langle W \rangle_K\) and \(I_d\) is the \(d \times d\) identity matrix.
\end{theorem}

As we have said earlier, this version is not the simple generalization of the usual weighted Carleson Embedding Theorem in \cite{NaTrVo99}. This is due to the extra factor \((I_d + t \widetilde{\md}_I  \langle W \rangle_I^{-1})^{-1}\) that appears (twice) in the left-hand side of the conclusion. However, the constants that appear in the theorem don't depend on the dimension \(d\) or on the weight \(W\). The proof of the result also uses arguments that were previously discussed in Section 4. 

\begin{proof}
Let $t=1$.
We first have to introduce the Bellman function associated to the problem. For \( \fd \in \mathbb{C}^d, \Fd \in \mathbb{R}, \dw \in \mathcal{M}_d(\mathbb{C}), \md \in \mathcal{M}_d(\mathbb{C}) \) satisfying
\begin{equation}\label{domain2}
\langle \dw^{-1} \fd, \fd \rangle_{\mathbb{C}^d} \leq \Fd \qquad \mbox{and} \qquad \md \leq \dw, 
\end{equation}
define the function \(\fb : \mathbb{C}^d \times \mathbb{R} \times \mathcal{M}_d(\mathbb{C}) \times \mathcal{M}_d(\mathbb{C})\) by
\[\fb(\fd, \Fd, \dw, \md):= 4 \big( \Fd - \big \langle (\dw + \md)^{-1} \fd, \fd \big \rangle_{\mathbb{C}^d} \big). \]

The Bellman function \(\fb\) has the following properties:
\begin{enumerate}[(i)]
    \item (Domain) The domain \(\mathfrak{D}:=\mathrm{Dom}\, \fb\) is given by \eqref{domain2}. 
   \item (Range) \(0 \leq \fb(\fd, \Fd, \dw, \md) \leq  4 \Fd \) for all \((\fd, \Fd, \dw, \md) \in \mathfrak{D}.\)
    \item (Concavity condition) Consider all tuples \(A=(\fd, \Fd, \dw, \md), A_+=(\fd_+, \Fd_+, \dw_+, \md_+)\) and \(A_-=(\fd_-, \Fd_-, \dw_-, \md_-)\) in \(\mathfrak{D}\) such that \(\fd=(\fd_+ + \fd_-)/2, \Fd=(\Fd_+ + \Fd_-)/2, \dw=(\dw_+ + \dw_-)/2,\) and \(\md = m + (\md_+ + \md_-)/2 = m + \widetilde{\md},\) where \(m\) is a positive definite matrix. For all such tuples, we have the following concavity condition:
        \[\fb(A) - \frac{\fb(A_+)+\fb(A_-)}{2} \geq \frac{1}{2} \big \langle (\dw + \widetilde{\md})^{-1} m (\dw + \widetilde{\md})^{-1} \fd, \fd \big \rangle.\]
  \end{enumerate}
\par
Let us now explain these properties of the function \(\fb\). 

The inequality \(\|\dw^{-1/2} \fd\|_{\mathbb{C}^d}^2 \leq \Fd \) follows from the Cauchy-Schwarz Inequality. The other inequality in \eqref{domain} is related to the Carleson condition.

Property (ii) follows trivially from the definition of \(\fb\).

To prove the concavity condition, we consider three tuples \(A, A_+, A_- \in \mathfrak{D}\) such that \(\fd=(\fd_+ + \fd_-)/2, \Fd=(\Fd_+ + \Fd_-)/2, \dw=(\dw_+ + \dw_-)/2,\) and \(\md = m + (\md_+ + \md_-)/2 = m + \widetilde{\md}\). Let \(\tilde{A} = (\fd, \Fd, \dw, \widetilde{\md})\). We prove the inequality in (iii) by splitting it into two inequalities. The first one, 
\[\fb(\tilde{A}) - \frac{1}{2} \big( \fb(A_+) + \fb(A_-) \big) \geq 0,\]
follows from the convexity of the first inequality in \eqref{domain} (like in the first part of the proof of Lemma \ref{extdom}). 
The second inequality, 
\[\fb(A) - \fb(\tilde{A})  \geq \frac{1}{2} \langle (\dw + \widetilde{\md})^{-1} m (\dw + \widetilde{\md})^{-1} \fd, \fd \rangle,\]
is obtained by showing that 
\[ (\dw + \widetilde{\md})^{-1} -  (\dw + \widetilde{\md} + m)^{-1} \geq \frac{1}{2} (\dw + \widetilde{\md})^{-1} m (\dw + \widetilde{\md})^{-1}. \]
To see this, notice that the left-hand side of this inequality can be written as
\[(\dw + \widetilde{\md})^{-1/2} \bigg ( I_d - \Big ( I_d + (\dw + \widetilde{\md})^{-1/2} m (\dw + \widetilde{\md})^{-1/2} \Big )^{-1} \bigg ) (\dw + \widetilde{\md})^{-1/2}.\]
If \(E := (\dw + \widetilde{\md})^{-1/2} m (\dw + \widetilde{\md})^{-1/2}\), we have that \(0 < E \leq I_d\), since \(m \leq \dw \leq \dw + \widetilde{\md}\). Then the inequality \(I_d - (I_d + E)^{-1} \geq \frac{1}{2} E\) is equivalent to \(I_d + E - I_d \geq \frac{1}{2} E(I_d + E)\), which can be rewritten as \(E \geq E^2\). This last inequality is clearly true since \(E \leq I_d\). It follows that 
\begin{align*}
(\dw + \widetilde{\md})^{-1} -  (\dw + \widetilde{\md} + m)^{-1}  & = (\dw + \widetilde{\md})^{-1/2} \big(I_d - (I_d + E)^{-1}\big) (\dw + \widetilde{\md})^{-1/2} \\
& \geq \frac{1}{2}  (\dw + \widetilde{\md})^{-1/2} E  (\dw + \widetilde{\md})^{-1/2} \\
& = \frac{1}{2} (\dw + \widetilde{\md})^{-1} m (\dw + \widetilde{\md})^{-1},
\end{align*}
which is the desired inequality. 

To prove Theorem \ref{matrix_CET}, let \(W\) be a matrix weight, \(f \in \ltr\) and \(\{A_I\}_I\) be a sequence of \( d \times d\) positive definite matrices. For any \(I \in \dd\), let
\[\fd_I = \langle W^{1/2} f \rangle_I \in \mathbb{C}^d, \quad \Fd_I = \langle \|f\|^2 \rangle_I \in \mathbb{R}, \]
\[\dw_I = \langle W \rangle_I \in \mathcal{M}_d(\mathbb{C}), \quad \md_I = \frac{1}{|I|} \sum_{J \subseteq I} \langle W \rangle_J A_J \langle W \rangle_J  \in \mathcal{M}_d(\mathbb{C}).\]
Then 
\[m_I = \frac{1}{|I|} \dw_I A_I \dw_I \quad \mbox{and} \quad \widetilde{\md}_I = \frac{1}{|I|} \sum_{J \subsetneq I} \langle W \rangle_J A_J \langle W \rangle_J .\]

For the interval \(I\), the concavity condition (iii) implies that
\begin{align*}
& \frac{|I|}{2} \big \langle (\dw_I + \widetilde{\md}_I)^{-1} m_I (\dw_I + \widetilde{\md}_I)^{-1} \fd, \fd \big \rangle \\
& \qquad \qquad \qquad \leq |I| \fb(\fd_I, \Fd_I, \dw_I, \md_I) - |I^+| \fb(\fd_{I^+}, \Fd_{I^+}, \dw_{I^+}, \md_{I^+}) -  |I^-| \fb(\fd_{I^-}, \Fd_{I^-}, \dw_{I^-}, \md_{I^-}).
\end{align*}
Iterating this inequality \(k\) times, we obtain
\begin{align*}
& \frac{|I|}{2} \sum_{\substack{
                    J \subseteq I\\
                    |J| > 2^{-k}|I| }} \langle (\dw_J+ \widetilde{\md}_J)^{-1} m_J (\dw_J + \widetilde{\md}_J)^{-1} \fd_J, \fd_J \rangle \\
& \qquad \qquad \qquad \leq |I| \fb(\fd_I, \Fd_I, \dw_I, \md_I) - 
\sum_{\substack{
                    J \subseteq I\\
                    |J| = 2^{-k}|I| }}
 |J| \fb(\fd_J, \Fd_J, \dw_J, \md_J) \\
& \qquad \qquad \qquad \leq |I| \fb(\fd_I, \Fd_I, \dw_I, \md_I) \leq  4 |I| \Fd_I.      
\end{align*}
Using that 
\[ \big \langle (\dw_I + \widetilde{\md}_I)^{-1} m_I (\dw_I + \widetilde{\md}_I)^{-1} \fd_I, \fd_I \big \rangle\\
= \frac{1}{|I|}  \big \langle (I_d +  \langle W \rangle_I ^{-1} \widetilde{\md}_I)^{-1} A_I  (I_d +  \langle W \rangle_I^{-1} \widetilde{\md}_I )^{-1} \langle W^{1/2} f \rangle_I, \langle W^{1/2} f \rangle_I \big \rangle,\]
and letting \(k \to \infty\), we get
\[\sum_{J \subseteq I} \big \langle (I_d +  \langle W \rangle_I ^{-1} \widetilde{\md}_J)^{-1} A_J  (I_d +  \widetilde{\md}_J  \langle W \rangle_I^{-1})^{-1} \langle W^{1/2} f \rangle_J, \langle W^{1/2} f \rangle_J \big \rangle_{\mathbb{C}^d} \leq 8 |I| \langle \|f\|^2 \rangle_I,\]
which is our desired conclusion for \(t = 1\). 

For $0 < t < 1$, just replace $A_I$ by $t A_I$ and apply the inequality which we have just proved.

\end{proof}

\begin{remark}
\normalfont
While this paper was prepared for publication, A. Culiuc and S. Treil posted a result which appears to be the correct generalization of the scalar weighted Carleson Embedding Theorem to matrix weights in finite dimension $d$ (see \cite{CuTr15}). 
In the notation
of Theorem \ref{matrix_CET}, it says that 
\[
\sd \big \langle A_I  \langle W^{1/2} f \rangle_I, \langle W^{1/2} f \rangle_I \big \rangle_{\mathbb{C}^d} \leq C(d)  \|f\|^2_{\ltr}\]
if
\[\frac{1}{|I|} \sum_{J \subseteq I} \langle W \rangle_J  A_J \langle W \rangle_J  \leq \langle W \rangle_I  \quad \text{for all }  I \in \dd,\]
(Theorem 1.2 in \cite{CuTr15}). An important step in their proof, the estimate (2.5) in \cite{CuTr15}, is essentially identical with our Theorem \ref{matrix_CET}, obtained with a different proof. 
\end{remark}

\bibliographystyle{plain}

\begin{bibsection}
\begin{biblist}

\bib{Ba91}{article}{
  author = {K. Ball},
  title = {The plank problem for symmetric bodies},
  journal = {Invent. Math.},
  volume = {104},
  year = {1991},
  number = {3},
  pages = {535-543},
  owner = {Andrei},
  timestamp = {2015.08.26}
}


\bib{BiPeWi14}{article}{
  author = {K. Bickel},
  author = {S. Petermichl}, 
  author = {B. Wick},
  title = {Bounds for the {H}ilbert transform with matrix {$A_2$} weights},
  journal = {J. Funct. Anal.},
  year = {2016},
  volume = {270},
  pages = {1719-1743},
  number = {5},
  owner = {Andrei},
  timestamp = {2016.05.02}
}

\bib{CoFe74}{article}{
  author = {R. R. Coifman},
  author = {C. Fefferman},
  title = {Weighted norm inequalities for maximal functions and singular integrals},
  journal = {Studia Math.},
  year = {1974},
  volume = {51},
  pages = {241-250},
  owner = {Andrei},
  timestamp = {2013.02.28}
}

\bib{CuTr15}{article}{
author={A. Culiuc},
author={S. Treil},
title={The Carleson Embedding Theorem with Matrix Weights},
year = {2015},
volume={Preprint, arXiv:1508.01716},
owner = {Andrei},
timestamp = {2015.06.29}

}



\bib{DaDo07}{article}{
  author = {K. R. Davidson},
  author = {A. P. Donsig},
  title = {Norms of {S}chur multipliers},
  journal = {Illinois J. Math.},
  year = {2007},
  volume = {51},
  pages = {743-766},
  number = {3},
  owner = {Andrei},
  timestamp = {2013.03.01}
}

\bib{Go03}{article}{
  author = {M. Goldberg},
  title = {Matrix {$A_p$} weights via maximal functions},
  journal = {Pacific J. Math.},
  year = {2003},
  volume = {211},
  pages = {201-220},
  number = {2},
  owner = {Andrei},
  timestamp = {2015.06.29}
}

\bib{HaHy14}{article}{
  author = {T. S. H{\"a}nninen},
  author = {T. P. Hyt{\"o}nen},
  title = {Operator-valued dyadic shifts and the {$T(1)$} theorem},
  year = {2016},
  volume = {180},
  pages={213 -- 253}
  number={2}
  journal={Monatsh. Math. }
  owner = {Andrei},
  timestamp = {2015.08.21}
  }

\bib{HuMuWh73}{article}{
  author = {R. A. Hunt},
  author = {B. Muckenhoupt}, 
  author = {R. L. Wheeden},
  title = {Weighted norm inequalities for the conjugate function and {H}ilbert
	transform},
  journal = {Trans. Amer. Math. Soc.},
  year = {1973},
  volume = {176},
  pages = {227-251},
  owner = {Andrei},
  timestamp = {2013.02.28}
}

\bib{Hy11}{article}{
  author = {T. P. Hyt{\"o}nen},
  title = {Representation of singular integrals by dyadic operators, and the
	{$A_2$} theorem},
  year = {2011},
  volume = {Preprint, arXiv:1108.5119},
  journal={Lecture notes of an intensive course at Universidad de Sevilla, Summer 2011},
  owner = {Andrei},
  timestamp = {2013.03.01}
}

\bib{Hy12a}{article}{
  author = {T. P. Hyt{\"o}nen},
  title = {The sharp weighted bound for general {C}alder\'{o}n-{Z}ygmund operators},
  journal = {Ann. of Math. (2)},
  year = {2012},
  volume = {175},
  pages = {1473-1506},
  number = {3},
  owner = {Andrei},
  timestamp = {2013.03.01}
}


\bib{hpv}{article}{
 author = {T. Hyt\"onen},
  author = {S. Petermichl},
  author =  {A. Volberg},
  title = {The sharp square function estimate with matrix weight},
  journal = {},
  year = {2017},
  volume = {Preprint,  arXiv:1702.04569},
  pages = {},
  number = {},
  owner = {Andrei},
  timestamp = {2015.06.29}
}




\bib{Is15}{article}{
author={J. Isralowitz},
title={A matrix weighted T$_1$ theorem for matrix kernelled CZOs and a matrix weighted John-Nirenberg theorem},
year={2015},
volume={Preprint, arXiv:1508.02474}
owner = {Andrei},
timestamp = {2015.08.21}

}

\bib{IHP}{article}{
author={J. Isralowitz}, 
author={H.-K. Kwon}, 
author={S.Pott},
title={Matrix-weighted norm inequalities for commutators and paraproducts with matrix symbols},
year={2015},
volume={Preprint, arXiv:1507.04032}
 journal = {to appear in J. London Math. Soc.},
}




\bib{hunt}{article}{
  author = {F. Nazarov},
  author = {S. Treil},
  title = {The hunt for a Bellman function: applications to estimates for singular integral operators and to other classical problems of harmonic analysis},
  journal = {St. Petersburg Math. J.},
  year = {1997},
  volume = {8},
  pages = {721--824},
  number = {5},
  owner = {Andrei},
  timestamp = {}
}


\bib{NaTrVo99}{article}{
  author = {F. Nazarov},
  author = {S. Treil},
  author =  {A. Volberg},
  title = {The {B}ellman functions and two-weight inequalities for {H}aar multipliers},
  journal = {J. Amer. Math. Soc.},
  year = {1999},
  volume = {12},
  pages = {909-928},
  number = {4},
  owner = {Andrei},
  timestamp = {2015.06.29}
}

\bib{ntvp}{article}{
author = {F. Nazarov},
author = {S. Petermichl},
author = {S. Treil},
author =  {A. Volberg},
  title = {Convex body domination and weighted estimates with matrix weights },
  volume = {arXiv:1701.01907},
  pages = {},
  number = {},
  year={2017},
  owner = {Andrei},
  timestamp = {2015.06.29}
}

\bib{Pe07}{article}{
  author = {S. Petermichl},
  title = {The sharp bound for the {H}ilbert transform on weighted {L}ebesgue
	spaces in terms of the classical {$A_p$} characteristic},
  journal = {Amer. J. Math.},
  year = {2007},
  volume = {129},
  pages = {1355-1375},
  number = {5},
  owner = {Andrei},
  timestamp = {2013.02.28}
}

\bib{PePo02}{article}{
  author = {S. Petermichl},
  author = {S. Pott},
  title = {An estimate for weighted {H}ilbert transform via square functions},
  journal = {Trans, Amer. Math. Soc.},
  year = {2002},
  volume = {354},
  pages = {1699-1703},
  number = {4},
  owner = {Andrei},
  timestamp = {2015.08.21}
}

\bib{PeVo02}{article}{
  author = {S. Petermichl},
  author =  {A. Volberg},
  title = {Heating of the {A}hlfors-{B}eurling operator: weakly quasiregular
	maps on the plane are quasiregular},
  journal = {Duke Math. J.},
  year = {2002},
  volume = {112},
  pages = {281-305},
  number = {2},
  owner = {Andrei},
  timestamp = {2013.02.28}
}



\bib{Pi12}{article}{
  author = {G. Pisier},
  title = {Grothendieck's theorem, past and present},
  journal = {Bull. Amer. Math. Soc. (N.S.)},
  year = {2012},
  volume = {49},
  pages = {237-323},
  number = {2},
  owner = {Andrei},
  timestamp = {2013.03.01}
}

\bib{PoSt17}{article}{
author={S. Pott}, 
author={A. Stoica}, 
title={Sharp bounds and \(T1\) theorem for Calder\'{o}n-Zygmund operators with matrix kernel on matrix weighted spaces},
year={2017},
volume={Preprint, arXiv:1705.06105},
 journal = {},
}

\bib{RoVa73}{book}{
  title = {Convex functions},
  publisher = {Academic Press},
  year = {1973},
  author = {A. W. Roberts},
  author = {D. E. Varberg},
  volume = {57},
  series = {Pure and Applied Mathematics},
  address = {New York-London},
  owner = {Andrei},
  timestamp = {2013.03.01}
}

\bib{Tr11}{inproceedings}{
  author = {S. Treil},
  title = {Sharp {$A_2$} estimates of {H}aar shifts via {B}ellman function},
  booktitle ={Recent trends in
Analysis, Theta Ser. Adv. Math.}, 
  pages={187-- 208},
  publisher={ Theta, Bucharest} 
  year = {2013},
  note = {arXiv:1105.2252},
  owner = {Andrei},
  timestamp = {2013.03.01}
}

\bib{TrVo97}{article}{
  author = {S. Treil},
  author = {A. Volberg},
  title = {Wavelets and the {A}ngle between {P}ast and {F}uture},
  journal = {J. Funct. Anal.},
  year = {1997},
  volume = {143},
  pages = {269-308},
  number = {2},
  owner = {Andrei},
  timestamp = {2015.06.29}
}

\bib{Vo97}{article}{
  author = {A. Volberg},
  title = {Matrix {$A_p$} weights via $S$-functions},
  journal = {J. Amer. Math. Soc.},
  year = {1997},
  volume = {10},
  pages = {445-466},
  number = {2},
  owner = {Andrei},
  timestamp = {2015.06.29}
}
  
\bib{Wi00}{article}{
  author = {J. Wittwer},
  title = {A sharp estimate on the norm of the martingale transform},
  journal = {Math. Res. Lett.},
  year = {2000},
  volume = {7},
  pages = {1-12},
  number = {1},
  owner = {Andrei},
  timestamp = {2013.02.28}
}

\end{biblist}
\end{bibsection}

\end{document}